\documentclass[12pt]{amsart}
\usepackage{graphicx} % Required for inserting images
\usepackage{amsmath,amsfonts,amssymb,amsthm}
\usepackage{mathtools}
\usepackage{tikz-cd}
\usepackage{tikz}
\usepackage{circuitikz}
\usepackage[utf8]{inputenc}
\usepackage{hyperref}
\usepackage{enumerate}
\usepackage{mathrsfs}
\usepackage{indentfirst}

\usepackage[margin=1.0in]{geometry}
\newtheorem{theorem}{Theorem}[section]

\newtheorem{proposition}[theorem]{Proposition}
\newtheorem{lemma}[theorem]{Lemma}
\newtheorem{corollary}[theorem]{Corollary}
\newtheorem{conjecture}[theorem]{Conjecture}

\theoremstyle{definition}
\newtheorem{definition}[theorem]{Definition}
\newtheorem{remark}[theorem]{Remark}
\newtheorem{example}[theorem]{Example}

\title{On the structure of locally conformally flat orbifolds and ALE metrics}
\author{Xiaokang Wang}
\address{Department of Mathematics, University of California, Irvine, CA 92697, USA}
\email{xiaokaw@uci.edu}

\begin{document}

%%%%%%%%%%%%%%%%%%%%%%%%%%%%%%%%%%%%%%%%%%%%%%%%%%%%%
\begin{abstract}
In this paper, we prove several structure theorems for locally conformally flat, positive Yamabe orbifolds and nonnegative scalar curvature, ALE manifolds. These two kinds of spaces can be related by conformal blow-up and conformal compactification. For the orbifolds, we prove that such orbifolds admit a manifold cover. For the ALE manifolds, the homomorphism of the fundamental group for the ALE space induced by the embedding of the ALE end is always injective.  Using these properties, several classifications of such ALE manifolds and orbifolds are given in low dimensions. As an application to the moduli space, we prove that the football orbifold $\mathbb{S}^4/\Gamma$ cannot be realized as the Gromov-Hausdorff limit. In addition, we prove the positive mass theorem of these ALE ends and give a simple proof for the optimal decay rate. Using the positive mass theorem, we can solve the orbifold Yamabe problem in the locally conformally flat case. 
    
\end{abstract}
\maketitle
\section{Introduction}
The locally conformally flat structure is a natural generalization of the conformal structure on two-dimensional surfaces. It is well known to modern geometric analysts because of its connection to the vanishing of the Weyl curvature. Much work has been done to study such manifolds, for example, \cite{SYconformal}. We will write LCF for short. In this paper, we always assume $n\geq 3$. We are interested in the moduli space of closed LCF manifolds with bounded curvature $L^{\frac{n}{2}}$ norm. To gain better control over the moduli space as it approaches the boundary, we assume that the manifolds in the moduli space have the Yamabe metric with a strictly positive Yamabe constant. In this way, we avoid local collapsing and ensure a non-collapsing orbifold limit. Such moduli spaces are considered in \cite{Aku}, \cite{TV1}, \cite{TV2}, \cite{TV3}. Similar moduli spaces are studied in the Einstein case; see, for example, \cite{And89}, \cite{BKN}.

Let $\mathfrak{M}(n, \mu_0, C_0)$ denote the set of closed n-dimensional Riemannian manifolds that satisfy the following:
\begin{itemize}
    \item locally conformally flat
    \item unit-scalar curvature:  $S(g)=1$
    \item $g$ is the Yamabe metric with the Yamabe constant $Y(M,[g])\geq \mu_0 >0$
    \item $\int_M |\operatorname{Ric(g)|^{\frac{n}{2}}}dvol_g < C_0$
\end{itemize}

Let $\mathfrak{M}'(n, \mu_0, C_0)$ be the moduli space if we in addition assume that the manifolds are orientable.

Note that the scalar curvatures are constant 1. Thus, the volume $\operatorname{vol}(M,g)$ is uniformly bounded from above by the Yamabe constant of $\mathbb{S}^n$ and from below by assumption. On the other hand, from the local ball non-collapsing, we have the uniform upper bound on the diameter; see \cite{Aku} for details.

\begin{definition}
A topological space $M^n$ is called an \textbf{orbifold} if for every point $x\in M$, there is a neighborhood $U\ni x$ where there exist $\tilde{U} \subseteq \mathbb{R}^n$ and a finite group $G$ acting on $\tilde{U}$ such that there is a $G$-invariant map $\pi: \tilde{U}\to U$ that induces a diffeomorphism $U\cong\tilde{U}/G$.

An orbifold $(M,g)$ is called a \textbf{Riemannian orbifold} if there is a $G$-invariant $C^\infty$ metric $\tilde{g}$ in $\tilde{U}$ such that away from singularities, $\pi^* g = \tilde{g}$.
\end{definition}

In this paper, we consider only the case where the singularities are isolated. This means that the action of $G$ contains only one fixed point in $\tilde{U}$. We say that such an orbifold is \textbf{oriented} if $M \setminus \{x_1, \dots, x_k\}$ is an oriented manifold. In this way, the orientation on $M \setminus \{x_1, \dots, x_k\}$ induces the orientation on $\mathbb{S}^{n-1}/\Gamma$, and thus the action of every orbifold group $\Gamma \subset {\rm{SO}}(n)$ is orientation-preserving with respect to this orientation. Thus, there is a well-defined choice of orientation on $M$.

For the moduli space $\mathfrak{M}(n, \mu_0, C_0)$, as shown in \cite{TV2}, it can be compactified by adding compact LCF multifolds of positive orbifold Yamabe invariant with finitely many singularities. A \textbf{Riemannian multifold} is obtained by identifying finitely many points from
\begin{align*}
    \tilde{M}=\coprod_{i=1}^N M_i,
\end{align*}
where $M_i$ is a Riemannian orbifold with positive orbifold Yamabe invariant. Note that orbifold points with different tangent cones can be identified together. We call a multifold singularity \textbf{irreducible} if it is the orbifold point in the usual sense, i.e., no other orbifold points are identified with this orbifold point.

Such non-trivial multifold points correspond to the points where curvature blows up in a sequence approaching the boundary of the moduli space. To understand the curvature blow-up behavior, a rescaling argument is needed. An important result is shown in \cite{TV3}, which demonstrates that such a limit always exists with Euclidean volume growth and, in particular, is an LCF, scalar-flat, asymptotically locally Euclidean (ALE) manifold:

\begin{definition}\label{def ALE} Let $(M,g)$ be a complete, non-compact Riemannian manifold. An \textbf{end} of $(M,g)$ is an unbounded component of the complement of some compact subset $K\subseteq M$. 

An end $E\subseteq M- K$ is called 
    \textbf{asymptotically locally Euclidean (ALE)} of order $\tau >0$ if there is a diffeomorphism:
\begin{align*}
     \phi: E\to (\mathbb{R}^n\setminus B(0,r))/\Gamma
\end{align*}
for some finite subgroup $\Gamma$ acting freely on $\mathbb{R}^n-B(0,r)$. Under such identification:
\begin{align*}
     &g_{ij}=\delta_{ij}+O(r^{-\tau})\\
     &\partial^{m}g_{ij}=O(r^{-\tau-|m|})
\end{align*}
for all partial derivatives of order $|m|$ with multi-index $m$ as $r\to\infty$. We call an end \textbf{asymptotically Euclidean (AE)} if the group $\Gamma$ is trivial.

We call $(M,g)$ an \textbf{ALE manifold} if it contains finitely many ALE ends $\{E_i\}$ of order $\{k_i\}$, respectively. We call $(M,g)$ an \textbf{AE manifold} if all the groups in the ends are trivial. If $(M,g)$ is orientable, then the orientation will induce an orientation on $\mathbb{S}^{n-1}/\Gamma$, and $\Gamma\subset{\rm{SO}}(n)$ with respect to this orientation.

\end{definition}

In this paper, our aim is to study such irreducible orbifolds and ALE manifolds. 
Let $(M,g)$ be a complete, non-compact, LCF, nonnegative scalar curvature, ALE manifold. It is well known that $(M,g)$ can be conformally compactified to be an LCF orbifold $M'$ with positive orbifold Yamabe invariant; see \cite{TV2}. Conversely, given an orbifold $(M',g')$ with at most isolated singularities and positive orbifold Yamabe invariant, we can associate such an orbifold with an ALE manifold by conformal blow-up using the conformal Green's function at these singularities. We will review these facts and give short proofs in Section~\ref{good orbifolds proof}.
Thus, all results in this paper can be equivalently stated for orbifolds or for ALE metrics.

\subsection{Orbifold results}

Any orbifold $M$ has an orbifold universal cover $\tilde{M}_{orb}$ in the sense that any orbifold is covered by the universal orbifold cover as a topological space, and $\pi_1^{orb}$ is defined as the deck transformation group of $\pi: \tilde{M} \rightarrow M$. The precise definition of $\pi_1^{orb}(M)$ can be found in \cite{orbifold}. In general, $\pi_1^{orb}(M)$ is not isomorphic to the topological fundamental group $\pi_1(M)$.

An important example is as follows:
\begin{definition}\label{Def football}
    A \textbf{football orbifold} is  $\mathbb{S}^n/\Gamma$, where $\Gamma \subseteq {\rm{O}}(n) \subset {\rm{O}}(n+1)$, having two fixed antipodal points. The orbifold metric is given by the quotient of the standard spherical metric. 
    %The Green's function metric at both orbifold points is the Schwarzschild ALE metric with group $\Gamma$.    
\end{definition}
In this case, $\pi_1^{orb}(M) \simeq \Gamma$, while $\pi_1(M) = \{e\}$ since it is simply connected.

We call an orbifold \textbf{good} if it is obtained from a manifold quotient by some group with fixed points. We say such an orbifold has a manifold cover, although it is not a genuine covering space. Football orbifolds are good orbifolds. Not all orbifolds are good; a famous example is the ``Teardrop" orbifold, which is constructed by identifying the boundary of two-dimensional disc $B_1\subseteq\mathbb{R}^2$ with the boundary of $B_1/\mathbb{Z}_k$. A Teardrop orbifold cannot be covered by a manifold. %Theorem \ref{group injective} also implies the following. 
The ``bad'' Teardrop example is $2$-dimensional, LCF, and admits a metric of positive curvature. But in higher dimensions we have the following.
\begin{theorem}\label{good orbifold}  If $(M,g)$ is a compact LCF orbifold with positive scalar curvature and $\dim(M) \geq 3$, then $(M,g)$ is a good orbifold.
\end{theorem}
Theorem \ref{good orbifold} is proved in Section \ref{good orbifolds proof}.

We next recall some background in conformal geometry. For an LCF manifold with non-negative scalar curvature, $(M,g)$, Schoen-Yau \cite{SYconformal} proved there is a developing map, $\phi:(\tilde{M},\tilde{g})\to (\mathbb{S}^n, g_{\mathbb{S}^n})$, which is a conformal embedding, where $\pi:\tilde{M}\to M$ is the universal covering of $M$, $\tilde{g} = \pi^*g$, and $g_{\mathbb{S}^n}$ is the round metric. This is generalized to good orbifolds in the following manner. 
For an LCF good Riemannian orbifold with positive scalar curvature and finitely many singular points, the orbifold universal covering $\tilde{M}_{orb}$ is a manifold by the universal property. Furthermore, $\pi^* g$ is a smooth metric, so the developing map is defined.
We write $M = \Omega/G$, where $\Omega\subseteq \mathbb{S}^n$, and $G\subseteq C(n) \leq \rm{O}(n+1,1)$, the conformal group, which acts properly and discontinuously on $\Omega$. Denote $\Lambda := \mathbb{S}^n\setminus\Omega$. When $M$ is compact, $\Lambda$ coincides with the limiting set of $G$. By Selberg's lemma, there is a finite index torsion-free subgroup $H\subseteq G$. Thus, $M$ has a finite manifold cover. Note the Hausdorff dimension $dim_\mathcal{H}(\Lambda_G) = dim_\mathcal{H}(\Lambda_H)$.

Now we define the conformal connected sum. Let $(M_1,g_1)$, $(M_2,g_2)$ be two LCF orbifolds. Fix an orientation on $\mathbb{S}^n$. Let $A\subseteq \mathbb{S}^n$ be diffeomorphic to $[0,1]\times\mathbb{S}^{n-1}$, where $\mathbb{S}^n\setminus A$ has two components, $D_1$, $D_2$, and each $D_i$ is diffeomorphic to $B_{\mathbb{R}^n}(1)$.

\begin{definition}[\cite{Iz1}]\label{conformal connected sum}
    A \textbf{conformal connected sum} of $(M_1,g_1)$, $(M_2,g_2)$ is defined by:
    \begin{align*}
        (M_1 \setminus \psi_1(E_1), g_1)\bigcup_{\psi_1\circ\psi_1^{-1}|_{\psi_1(A)}}(M_2 \setminus \psi_2(E_2), g_2)
    \end{align*}
    where $E_i = A\cup D_i$ and $\psi_i: E_i \to M_i$ is the orientation-preserving conformal map. We denote it as $(M_1,g_1)\#_C(M_2,g_2)$.
\end{definition}

We can use the above definition to get a strong restriction on certain compact LCF orbifolds:
\begin{theorem}\label{classification} Let $(M,g)$ be a compact LCF orbifold with $\dim_{\mathcal{H}}(\Lambda) < 1$, then $(M,g)$ is finitely covered by $\mathbb{S}^n$ or a conformal connected sum
\begin{align*}
    \#_k \mathbb{S}^1\times \mathbb{S}^{n-1},
\end{align*}
for some $k>0$. %, where when $k=0$, $(M,g)$ is finitely covered by $\mathbb{S}^n$.

If $n = 3,4$, then $(M,g)$ is a conformal connected sum 
\begin{align*}
M = \#_i  (\mathbb{S}^{n-1} \times \mathbb{R}/G_i) \#_j (\mathbb{S}^n/\Gamma_j),
\end{align*}
where each $G_i$ is a discrete cocompact subgroup of the isometry group of the standard product metric of $\mathbb{S}^{n-1}\times \mathbb{R}$ with at most isolated singularities; each $\Gamma_i$ is a finite subgroup of ${\rm{O}}(n+1)$ with at most isolated singularities. 

In any dimension: If $(M,g)$ is orientable, then if $n$ is odd, $(M,g)$ is a manifold; if $n$ is even, the number of non-trivial orbifold points is even, and they occur in orientation-reversing conjugate pairs. 
%of football orbifolds and $\mathbb{S}^1\times\mathbb{S}^3/\Gamma$. 
\end{theorem}
Theorem \ref{classification} will be proved in Section \ref{classification proof}.

\begin{remark}
    Theorem \ref{classification} is an extension of Theorem 6.1 in Izeki \cite{Ilimit} to the orbifold case. Here we make use of the structure of the orbifolds proved in Theorem \ref{good orbifold}.
\end{remark}

\begin{remark}\label{n=3,4} If $n = 3$ or $4$, then non-negative scalar curvature implies $\dim_{\mathcal{H}}(\Lambda) < 1$. Thus, the above theorem directly applies to the LCF orbifold with non-negative scalar curvature.
\end{remark}

\begin{remark}
    In odd dimensions, by Theorem \ref{good orbifold}, the only non-trivial orbifold points are $\mathbb{Z}_2$-singularities, which are obtained by taking the quotient of an orientation-reversing $\mathbb{Z}_2$ map. In even dimensions, we will need the algebraic result about the orbits of fixed points proved in the Appendix.
\end{remark}

%This will also hold in higher dimensions, with an extra assumption about the Hausdorff dimension of limit set, see below. 

%Using this structure, we can prove the following compactification theorem.

%Let $(M,g)$ be a complete, non-compact, LCF, scalar-flat, ALE manifold. Then it is well known that $(M,g)$ can be conformally compactified by an %LCF orbifold $M'$ with positive orbifold Yamabe invariant; see \cite{TV}.

%Using the conformal Green's function as the conformal factor, we get the optimal ALE rate for these ALE manifolds, this answers a question in \cite{TV1}.

%With this correspondence, we can prove a low dimensional classification theorem. 
%We say an ALE manifold is Schwarzchild ALE if it is finitely covered by a Schwarzchild manifold. The conformal compactification of the Schwarzschild metric with 2 ends is called a football orbifold. 
%Define football orbifold here, and some words about conformal connect sum. 
%\begin{theorem}Let $(M,g)$ be an LCF orientable orbifold with $\dim_{\mathcal{H}}(\Lambda) < 1$.
%%If $n=3$, then $(M,g)$ is a manifold which is a conformal connect sum of LCF positive scalar manifolds
%%%If $n = 4$, then $(M,g)$ is a conformal connect sum of LCF positive scalar manifolds 
%%and football orbifolds.  
%Then $(M,g)$ is an orbifold conformal connect sum of football orbifolds.  
%\end{theorem}

If $n = 4$,  Chen-Tang-Zhu \cite{ChenZhu} show that $(M,g)$ is diffeomorphic to a connected sum of football orbifolds. The key point is that, in this situation, we have a \textit{conformal} connected sum. As a consequence of Theorem \ref{classification}, we have the following corollary:
\begin{corollary}\label{isotropy}
    Any compact 4-orbifold $(M,g)$ with at most isolated
singularities and with positive isotropic curvature has a finite
manifold cover that is diffeomorphic to $\#_k (\mathbb{S}^1 \times \mathbb{S}^3)$.
\end{corollary}
\begin{remark}
    The corollary a priori seems to have nothing to do with LCF metrics. The connection is pointed out by \cite{ChenZhu}, namely: a compact four-orbifold with at most isolated singularities admits positive isotropic curvature if and only if it admits an LCF metric with positive scalar curvature (Corollary 2 in \cite{ChenZhu}). Since the conclusion is topological, we work on its LCF metric. 
\end{remark}

%\begin{theorem}\label{classification}
%    The orientable, LCF, scalar-flat, ALE manifolds with finitely many ends in dimension 3 and 4 are diffeomorphic to the connected sum of %Schwarzchild ALE and the LCF, positive scalar manifolds and blow-up with finitely many points using the conformal Green's function.
%\end{theorem}

%An interesting observation is that the non-trivial ALE ends in these cases always appear in pairs. We think this phenomenon could also be true in high dimension.

\subsection{ALE results}

Our first result for the ALE manifolds is the structure of the ALE ends. We denote: the k-th end $E_k$ is diffeomorphic to $\mathbb{R}^+\times\mathbb{S}^{n-1}/\Gamma_k$ for some finite subgroup $\Gamma_k\subset {\rm{O}}(n)$.

\begin{theorem}\label{group injective}
    Let $(M,g)$ be a complete, non-compact, LCF, non-negative scalar curvature ALE manifold. Then the inclusion map $i_k:E_k \to M$ induces an injective homomorphism ${i_k}_*:\Gamma_k\to\pi_1(M)$.
\end{theorem}

Theorem \ref{group injective} will be proved in Section \ref{isotropy ALE section}. 

Note that this is not always true for ALE manifolds without the LCF assumption. For example, in the Eguchi-Hanson metric, $M_{EH}$, the group in the end is $\Gamma\cong \mathbb{Z}_2$, but $M_{EH}$ is simply connected. We prove this theorem by combining the embedding theorem of \cite{SYconformal} and the observation that the only closed totally umbilic submanifolds in $\mathbb{R}^n$ are $\mathbb{S}^{n-1}$. 

From Remark \ref{n=3,4}, in dimension 3 and 4, we know that LCF orbifold with positive Yamabe satisfies $\operatorname{dim}_\mathcal{H}(\Lambda)<1$. Thus, Theorem \ref{classification} an be applied directly to give an ALE manifold classification. In particular, we have the following corollary about the ends of oriented ALE manifolds.
\begin{corollary} Let $M$ be an orientable, LCF, ALE manifold with non-negative scalar curvature. If $n = 3$, then $M$ is AE. If $n=4$, the groups of the non-trivial ALE ends occur in orientation-reversing conjugate pairs, so there must be an even number of ends with non-trivial group. In particular, any orientable LCF ALE $4$-manifold with only one end is AE. 
\end{corollary}

As an application, We have a refined structure about the boundary of the oriented moduli space.

\begin{theorem}\label{t:even}
    Let $(M_i, g_i)$ be a sequence in  $\mathfrak{M}'(n, \mu_0, C_0)$.
    If $n=3$, then the Gromov-Hausdorff limit is finitely many LCF manifolds with finitely many points identified. 

If $n=4$, then the Gromov-Hausdorff limit is finitely many LCF orbifolds with finitely many points identified, such that the tangent cone at any singular point has an even number of cones with non-trivial orbifold groups, which appear in pairs. 
\end{theorem}
Theorem \ref{t:even} will be proved in Section \ref{moduli space application proof}. This also relies on the result proved in the Appendix, giving certain conditions when orbifold isotropy groups must occur in pairs. For an illustration of the degeneration in Theorem~\ref{t:even}, see Figure \ref{fig:moduli}. In Figure \ref{fig:moduli}, the limit, $(M_\infty, g_\infty)$, will be a multifold with exactly one multifold singularity. The rescaling limit at the point where the curvature tends to infinity, $(\hat{M}_\infty, \hat{g}_\infty)$, will be a 2-end ALE manifold.

\begin{figure}[t]
    \centering
    \resizebox{1\textwidth}{!}{\begin{circuitikz}
\tikzstyle{every node}=[font=\Huge]
\draw [line width=2.0pt, short] (4.75,10) .. controls (5.5,5.75) and (14.25,6.25) .. (16.25,8.5);

\draw [line width=2.0pt, short] (4.75,10) .. controls (4,14.25) and (13.75,14.5) .. (16,13);
\draw [line width=2.0pt, short] (16,13) .. controls (14.25,13) and (14.5,8.5) .. (16.25,8.5);
\draw [line width=2.0pt, dashed] (16,13) .. controls (17.25,13) and (17.5,8.75) .. (16.25,8.5);
\draw [line width=2.0pt, short] (7.75,10.25) .. controls (10,9) and (11,9) .. (12.75,10.25);
\draw [line width=2.0pt, short] (8.25,10) .. controls (10.25,10.75) and (10.5,10.75) .. (12.25,10);
\draw [line width=2.0pt, short] (16,13) .. controls (15.75,22.75) and (27,21.25) .. (25.75,14);
\draw [line width=2.0pt, short] (16.25,8.5) .. controls (17.5,2.25) and (27.75,-0.25) .. (26,8);
\draw [line width=2.0pt, short] (18,11) .. controls (17,18) and (24,17.75) .. (24.25,14);
\draw [line width=2.0pt, short] (18,11) .. controls (18,7) and (23,4) .. (24.5,8.25);
\draw [line width=2.0pt, short] (24.5,8.25) .. controls (25.25,7.75) and (25.5,7.75) .. (26,8);
\draw [line width=2.0pt, short] (24.25,14) .. controls (25,14.5) and (25.25,14.5) .. (25.75,14);
\draw [line width=2.0pt, dashed] (24.5,8.25) .. controls (25.25,8.5) and (25.5,8.5) .. (26,8);
\draw [line width=2.0pt, dashed] (24.25,14) .. controls (25,13.5) and (25.25,13.5) .. (25.75,14);
\draw [line width=2.0pt, short] (24.25,14) -- (24.75,11.75);
\draw [line width=2.0pt, short] (25.75,14) -- (25.5,12);
\draw [line width=2.0pt, short] (24.5,8.25) -- (25,10.25);
\draw [line width=2.0pt, short] (26,8) -- (25.5,10.25);
\draw [line width=2.0pt, short] (24.75,11.75) .. controls (24.25,10.75) and (24.5,10.5) .. (25,10.25);
\draw [line width=2.0pt, short] (25.5,12) .. controls (26.5,11.25) and (26.5,10.25) .. (25.5,10.25);
\draw [line width=2.0pt, short] (25.25,11.5) .. controls (25.75,11) and (26,11) .. (25.25,10.5);
\draw [line width=2.0pt, short] (25.5,11.25) .. controls (25.25,11) and (25.25,10.75) .. (25.5,10.75);
\draw [->, >=Stealth] (28,12.25) -- (37.25,12.25);
\node [font=\LARGE] at (32,13) {Gromov-Hausdorff};
\draw [line width=2.0pt, short] (38.75,12.25) .. controls (38.75,8.5) and (50.25,9.25) .. (52,10);
\draw [line width=2.0pt, short] (38.75,12.25) .. controls (38,17.75) and (48,17.25) .. (51,15.75);
\draw [line width=2.0pt, short] (51,15.75) .. controls (49.75,14) and (50.75,10.75) .. (52,10);
\draw [line width=2.0pt, short] (42.5,12.75) .. controls (45,11.25) and (45.5,11.5) .. (47.5,12.75);
\draw [line width=2.0pt, short] (43.5,12.25) .. controls (45,12.75) and (45.25,12.75) .. (46.5,12.25);
\draw [line width=2.0pt, dashed] (51,15.75) .. controls (52.5,13.25) and (52.5,12) .. (52,10);
\draw [line width=2.0pt, short] (51,15.75) .. controls (52,23.5) and (67.5,25.5) .. (62.75,12.75);
\draw [line width=2.0pt, short] (52,10) .. controls (57.25,3.5) and (65.25,5) .. (62.75,12.75);
\draw [line width=2.0pt, short] (53.75,13) .. controls (53.75,18.75) and (61.75,19.5) .. (62.75,12.75);
\draw [line width=2.0pt, short] (53.75,13) .. controls (54.75,10.25) and (62.25,6.5) .. (62.75,12.75);
\draw [->, >=Stealth] (26,2) -- (36.5,-3.25)node[pos=0.5, fill=white]{blow up};
\draw [line width=2.0pt, short] (42,4.5) -- (44.75,-1.75);
\draw [line width=2.0pt, short] (47.75,-1.75) -- (50.75,4.75);
\draw [line width=2.0pt, short] (44.75,-1.75) .. controls (42.5,-3.5) and (44,-4.25) .. (44.75,-5);
\draw [line width=2.0pt, short] (44.75,-5) -- (41.5,-9.25);
\draw [line width=2.0pt, short] (47.75,-1.75) .. controls (49.75,-3.25) and (49.25,-4.25) .. (47.75,-4.75);
\draw [line width=2.0pt, short] (47.75,-4.75) -- (51.5,-9.75);
\draw [line width=2.0pt, short] (41.5,-9.25) .. controls (46.75,-11.75) and (47.25,-11.75) .. (51.5,-9.75);
\draw [line width=2.0pt, short] (42,4.5) .. controls (46.25,7) and (46.75,6.5) .. (50.75,4.75);
\draw [line width=2.0pt, dashed] (41.5,-9.25) .. controls (46.5,-8) and (46.75,-8.25) .. (51.5,-9.5);
\draw [line width=2.0pt, dashed] (42,4.5) .. controls (46.5,3.25) and (46.5,3.25) .. (50.75,4.75);
\draw [line width=2.0pt, short] (45,-3.5) .. controls (46.5,-4.5) and (46.75,-4.25) .. (47.75,-3.5);
\draw [line width=2.0pt, short] (45.75,-3.75) .. controls (46.5,-3) and (46.75,-3) .. (47.25,-3.75);
\node [font=\Huge] at (12.5,0.75) {};
\node [font=\Huge] at (13.5,1.5) {$(M_i, g_i)$};
\node [font=\Huge] at (56.5,4.25) {$(M_{\infty}, g_{\infty})$};
\node [font=\Huge] at (56.5,-11.25) {$(\hat{M}_{\infty}, \hat{M}_{\infty})$};
\end{circuitikz}}
    \caption{A sequence $\{(M_i,g_i)\}$, where $(M_i,g_i)\in \mathfrak{M}'(4,\mu_0, C_0)$. The limit $(M_\infty, g_\infty)$ is a multifold with one multifold singularity. The rescaled limit will be a 2-end ALE manifold $(\hat{M}_\infty, \hat{g}_\infty)$.}
    \label{fig:moduli}
\end{figure}
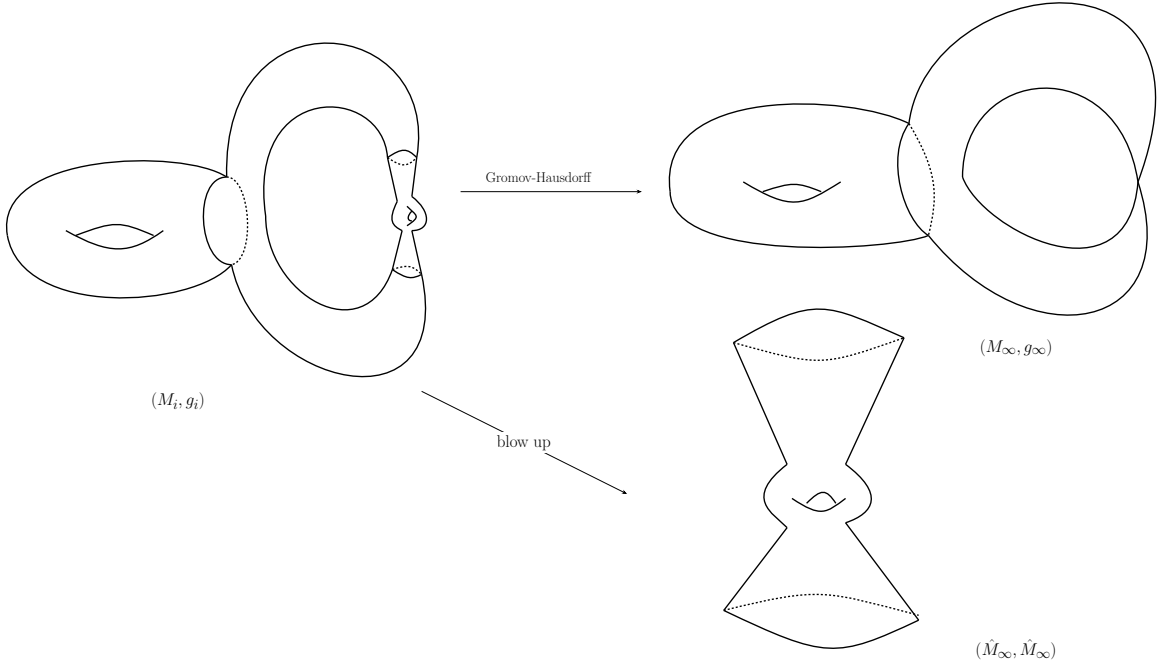

This has the following corollary that rules out certain multifolds being in the boundary of the oriented moduli space. 
\begin{corollary}\label{moduli space application}
    Let $n=3,4$, and $(M_o,g_o)$ be a closed, orientable, LCF multifold with finitely many isolated singularities and with only irreducible orbifold points. Then $(M_o,g_o)$ cannot be realized as a limit of a sequence of closed, orientable, LCF manifolds $(M_i^n, g_i)\in\mathfrak{M}'(n, \mu_0, C_0)$.
\end{corollary}

In particular, we can show that the football metric $\mathbb{S}^4/\Gamma$ is not in the boundary of the \textit{oriented} moduli space. However, for certain groups $\Gamma$, it is within the boundary of the \textit{non-orientable} moduli space; see Section \ref{nonorientable construction} for the construction. 

In the Einstein case, a similar result is shown in \cite{biquard}, \cite{ozuch2022noncollapsed}, where $\mathbb{S}^4/\Gamma$ cannot be realized as a limit of non-collapsing Einstein manifolds with $L^2$ curvature bound. In the Einstein case, they identified certain local obstructions on the singularity. In our case, we only use the structure of the ALE manifolds from Theorem \ref{classification} when rescaling the metric near the singularity.

Note that, similar to Definition \ref{def ALE}, we can define an ALE orbifold with isolated singularities, where we allow the compact region $K$ to have isolated singularities.
We can prove a positive mass theorem for such ALE orbifolds.
\begin{theorem}\label{PMT}
    Let $(M,g)$ be an LCF, nonnegative scalar curvature, ALE orbifold with $\tau > \frac{n-2}{2}$, then the ADM mass is non-negative for any end of $(M,g)$. The mass is zero at an end if and only if $(M,g)\cong (\mathbb{R}^n,g_{\mathbb{R}^n})/\Gamma$, where $\Gamma\subseteq {\rm{O}}(n)$ is a finite subgroup that acts isometrically and fixes the origin.
    Furthermore, if $(M,g)$ is scalar-flat, then $(M,g)$ is ALE of order $n-2$, which is optimal.
\end{theorem}
The optimal decay will be proved in Section \ref{good orbifolds proof} (See Corollary \ref{order of decay} and Remark \ref{remark of order of decay}). The positive mass part will be proved in Section \ref{PMT proof}. 

This is significant since the positive mass theorem for ALE manifolds with non-negative scalar curvature is not always true. See the counterexamples in \cite{Lebrun}. 
The optimal decay rate for obstruction-flat scalar-flat ALE metrics was previously proved by \cite{AcheViaclovsky}, but our proof in the LCF case is much easier by utilizing the conformal Green's function. 

%Finally, we have an application to the orbifold Yamabe problem in the LCF case in dimension 4:
%\begin{corollary}\label{orbifold Yamabe corollary}
%If $(M^4,g)$ is an LCF compact 4-orbifold with positive scalar curvature, 
%%If $(M,g)$ is not conformal to the football orbifold with quotient sphereical metric (Def. \ref{Def football}). 
%then there exists a solution $\tilde{g} = u^2 g$ to the orbifold Yamabe problem on $(M,g)$. 
%\end{corollary}

For application, we consider the orbifold Yamabe problem. Given a Riemannian orbifold $(M,g)$, one wants to find a constant scalar metric in the conformal class $[g]$. Unlike the manifold Yamabe problem, this is not always solvable. Counterexamples are constructed in \cite{viaclovsky2010monopole}. One core missing ingredient is the positive mass theorem for ALE manifolds. With the help of Theorem \ref{PMT}, we can solve the orbifold Yamabe problem in the special case.
\begin{corollary}\label{orbifold Yamabe corollary}
If $(M^n,g)$ is an LCF compact orbifold with positive scalar curvature, 
%If $(M,g)$ is not conformal to the football orbifold with quotient spherical metric (Def. \ref{Def football}). 
then there exists a solution $\tilde{g} = u^\frac{4}{n-2} g$ to the orbifold Yamabe problem on $(M,g)$. 
\end{corollary}
This will be proved in Section~\ref{PMT proof}.

%%%%%%%%%%%%%%%%%%%

 \textbf{Acknowledgments}: We thank Jeff Viaclovsky for patient guidance and helpful discussions. The author is supported by NSF Research Grant DMS-2105478.

%%%%%%%%%%%%%%%%%%%%%%%%%%%%%%%%%%%%%%%%%%%%%%%%%%%%%
\section{Preliminaries}
\subsection{Yamabe invariant}
Let us recall the Yamabe metric and Yamabe invariant:
\begin{definition}[Yamabe constant]
    Let $(M,g)$ be a compact Riemannian manifold. The \textbf{Yamabe constant} associated with its conformal class $[g]$ is
    \begin{align*}
        Y(M,[g]) = \inf_{\tilde{g}\in[g]}\operatorname{Vol}(\tilde{g})^{-\frac{n-2}{n}}\int_M S({\tilde{g})} dvol_{\tilde{g}}.
    \end{align*}
    The \textbf{Yamabe metric} is the metric in $[g]$ that achieves the above infimum.

    The \textbf{Yamabe invariant} of $M$ is the maximum of the Yamabe constant among all conformal classes $[g]$:
    \begin{align*}
        Y(M) = \sup_{[g]} Y(M,[g]).
    \end{align*}
\end{definition}
Note that $Y(M,[g])$ is scaling invariant. From the famous resolution of the Yamabe problem by \cite{Ya60, Tr68, Au76, Sch84}, such a Yamabe minimizer is always realized by a smooth metric with constant scalar curvature. Note that the Yamabe metric with a strictly positive Yamabe invariant implies the universal upper bound for the Sobolev constant, hence no local collapsing.

We also need a generalization of the Yamabe constant and Yamabe invariant on orbifolds. Denote the \textbf{orbifold Yamabe constant} of an orbifold $(M,g)$:
\begin{align*}
    Y_{orb}(M,[g]) = \inf_{\tilde{g}\in[g]}\operatorname{Vol}(\tilde{g})^{-\frac{n-2}{n}}\int_M S({\tilde{g}}) dvol_{\tilde{g}},
\end{align*}
and the \textbf{orbifold Yamabe invariant}:
\begin{align*}
    Y_{orb}(M) = \sup_{[g]} Y_{orb}(M,[g]).
\end{align*}
Parallel to the manifold Yamabe problem, an analog of Aubin's existence theorem is as follows:
\begin{theorem}[\cite{AB04},\cite{akutagawa2012computations}]\label{t:orbifold Yamabe existence}
Let $(M,g)$ be a Riemannian orbifold with isolated singularities $\{p_1, ...,p_k\}$, with orbifold groups $G_i \leq {\rm{O}}(n)$, for $i=1,...,k$. Then:
\begin{align*}
    Y_{orb}(M,[g])\leq Y(\mathbb{S}^n)\min_{i}|G_i|^{-\frac{2}{n}}.
\end{align*}
Furthermore, if this inequality is strict, then there exists a smooth conformal metric $\tilde{g} = u^{\frac{4}{n-2}}g$, which minimizes the Yamabe functional, and thus has constant scalar curvature.
\end{theorem}
To fully solve the orbifold Yamabe problem, we need to show the strict inequality. If one wants to use Schoen's test function from \cite{Sch84}, a positive mass theorem for ALE manifolds is needed. However, a positive mass theorem for ALE manifolds is not always true. See the counterexamples of \cite{Lebrun} for negative mass ALE metrics. This makes the orbifold Yamabe problem more subtle than in the manifold case. In fact, an example of the non-existence of the constant scalar metric in certain conformal classes is given in \cite{viaclovsky2010monopole}. See also \cite{ju2023conformally} for recent developments.

In this paper, we prove the orbifold Yamabe problem for LCF orbifolds with a positive orbifold Yamabe constant. See Corollary \ref{orbifold Yamabe corollary}.

%%%%%%%%%%%%%%%%%%%%%%%%%%%%%%%%%%%%%%%%%%%%%%%%%%%%%%

\subsection{Locally conformally flat manifolds and Kleinian structure}\label{s:lcf and Kleinian}
We are interested in the locally conformally flat (LCF) manifolds:
\begin{definition}[The LCF manifolds]\label{def 1}
    A Riemannian manifold $(M^n,g)$ is called \textbf{locally conformally flat (LCF)} if for every point $p\in M$, there exists a neighborhood $U$, a chart $\phi:\mathbb{R}^n\to U$ such that $\phi^*g(x)=\lambda(x)g_{\mathbb{R}^n}$, for some smooth function $\lambda:\mathbb{R}^n\to\mathbb{R}^+$.

\end{definition}

In order to understand the structure of such manifolds, we recall Schoen-Yau's embedding theorem from \cite{SYconformal}, which is later verified by \cite{ChodoshLi}, \cite{LUS}.
\begin{theorem}[Schoen-Yau]\label{thm SY}
    Let $(M^n,g)$ be an $n$-dimensional ($n\geq 3$) complete LCF manifold with scalar curvature $S(g)\geq 0$, then for the universal cover $(\tilde{M},g)$ there exists a conformal map $\phi:\tilde{M}\to \mathbb{S}^n$ which is an embedding.
\end{theorem}

Now, we introduce several concepts of Kleinian manifolds:

The conformal embedding map $\phi$ induces an injective homomorphism, which is called the \textbf{holonomy representation}, from the deck transformation of $\tilde{M}$, $\operatorname{Deck}(\tilde{M})\cong \pi_1(M)$, to the conformal group of $\mathbb{S}^n$, $C(n)$:
\begin{align*}
        \rho:\pi_1(M)\to C(n).
\end{align*}

Denote $\Omega\subseteq\mathbb{S}^n$ as \textbf{the domain of discontinuity}, which is the set in which $\rho(\pi_1(M))$ acts on properly discontinuously, and $\Lambda$ as \textbf{the limit set} of $\rho(\pi_1(M))$, which is the complement of $\Omega$ and is the minimal invariant closed subset of $\rho(\pi_1(M))$. A manifold $(M,g)$ is called \textbf{Kleinian} if $(M,g)$ is conformally equivalent to some $\Omega/\Gamma$. 

In addition, we introduce the Liouville theorem \cite{conformalbook}:
\begin{theorem}[Liouville]\label{thm L}
    Let $U$, $V$ be open connected subsets of $\mathbb{S}^n$, $n\geq 3$, $f:U\to V$ be a conformal map. Then $f$ can be uniquely extended to a conformal map $f:\mathbb{S}^n\to \mathbb{S}^n$.
\end{theorem}

From Theorem \ref{thm SY} and Theorem \ref{thm L}, we have that an LCF manifold with non-negative scalar curvature is Kleinian.

We also need to characterize the different elements in $C(n)$. Note that the group $C(n)$ has an extension to the action on the hyperbolic space $\mathbb{H}^{n+1}$ as isometries. In particular, $C(n)\cong\operatorname{Isom}(\mathbb{H}^{n+1})\leq {\rm{O}}(n+1,1)$, \cite{conformalbook}. If we consider the ball model, $(\mathbb{D}^{n+1},g_{\mathbb{H}^{n+1}})$, then by the Brouwer fixed point theorem, every element in $C(n)$ acts on $\mathbb{H}^{n+1}$ and will have fixed points, so we distinguish different conjugate classes by distinguishing their fixed points.
\begin{definition}\label{conjugate class}\cite{conformalbook}
    Let $g\in C(n)\cong \operatorname{Isom}(\mathbb{H}^{n+1})$:
    \begin{itemize}
            \item $g$ is called \textbf{elliptic} if $g$ has fixed points in $\mathbb{D}^{n+1}$;
            \item $g$ is called \textbf{parabolic} if $g$ has a single fixed point in $\partial\mathbb{H}^{n+1}\cong\mathbb{S}^n$;
            \item $g$ is called \textbf{hyperbolic} if $g$ has two distinct fixed points in $\partial\mathbb{H}^{n+1}\cong\mathbb{S}^n$.
    \end{itemize}
\end{definition}

\begin{remark}\label{elliptic}
     Note that if $g$ is of finite order, then $g$ is elliptic. Also, note that if $g$ is of infinite order, then the action $g$ on $\Omega$ is fixed point free.
\end{remark}

From \cite{conformalbook}, if $g\in C(n)$ is elliptic, then there exists $\gamma\in C(n)$ such that $\gamma^{-1} g\gamma\in {\rm{O}}(n+1)$. Later, we will use this to study the local isotropy subgroups.

A key property to identify $C(n)$ with a subgroup of ${\rm{O}}(n+1)$ is the famous Selberg Lemma.
\begin{theorem}[Selberg Lemma]\label{theorem.selberg}
    A finitely generated subgroup of a linear group over a field of characteristic 0 has a finite-index subgroup which is torsion-free.
\end{theorem}
In particular, we can use the fundamental result (Corollary 4.8, \cite{wehrfritz2012infinite}), there is a normal, finite-index, torsion-free subgroup. We should point it out that the construction of such normal subgroup is elementary: let $G$ be a finitely generated subgroup of linear group, and let $H\leq G$ be the finite-index subgroup produced by Theorem \ref{theorem.selberg}. Then, consider $N=\bigcap_{g \in G} g H g^{-1}$. Since $H$ is of finite-index, then the intersection is over a finite collection, which means $N$ is also of finte-index. Clearly, $N$ is normal and torsion-free. Thus, if $(M,g)$ is Kleinian, there exists a finite cover whose fundamental group is torsion-free.

%%%%%%%%%%%%%%%%%%%%%%%%%%%%%%%%%%%%%%%%%%%%%%%%%%%%%%%

\section{ALE ends and orbifold singularities}

\subsection{Totally umbilic submanifolds in \texorpdfstring{$\mathbb{R}^n$}{Rn}}

To study the conformal embedding, we need to understand the conformal invariant geometric objects. Motivated by this, we study the conformal invariant submanifolds. 

Let $N\subset (M,g)$ be a hypersurface. Consider the conformal change: $g\to \hat{g}= e^{2\psi} g$, then the second fundamental form of $N$:

\begin{align*}
    \hat{\text{II}}(X,Y)=\text{II}(X,Y)-\hat{g}(X,Y)\nabla\psi|_{normal}
\end{align*}
for every $X,Y$ smooth vector fields tangential to $N$.

\begin{definition}[Totally umbilic submanifolds]
    $N\subset(M,g)$ is called \textbf{totally umbilic} if $\forall x\in N$, 
    \begin{align*}
        {\text{II}}_x=\lambda(x) g|_N,
    \end{align*}
    for some smooth function $\lambda:N\to \mathbb{R}$.
\end{definition}
It can be easily seen that the totally umbilic submanifolds are invariant under conformal change. Note that this invariance is pointwise, i.e., the umbilic points are invariant under conformal change.

For $(M,g)$, if there is a conformal embedding $\phi: M \to \mathbb{S}^n$, and $(M,g)$ is not conformally equivalent to $(\mathbb{S}^n, g_{\mathbb{S}^n})$, then by the stereographic projection, we have a conformal embedding $\bar{\phi}:M \to \mathbb{R}^n$. Thus, all the totally umbilic submanifolds in $(M,g)$ are totally umbilic in $\mathbb{R}^n$, which gives a very strong constraint.

Recall the classical theorem of totally umbilic submanifolds in $\mathbb{R}^n$. This is an old theorem that dates back to Cartan. Since it is difficult to find the precise reference, we include the proof as well.

\begin{theorem}[Totally umbilic submanifolds in $\mathbb{R}^n$]\label{tus in Rn}
Let $n\geq 3$. The only closed totally umbilic submanifolds in $\mathbb{R}^n$ are the round $\mathbb{S}^{n-1}$.     
\end{theorem}
\begin{proof}
    Let $i: N\to \mathbb{R}^n$ be a totally umbilic submanifold. By the Gauss-Codazzi equation, we have:
    \begin{align*}
        0 = \operatorname{Rm}_N(X,Y,Z,W) - \langle \text{II}(X,W), \text{II}(Y,Z)\rangle + \langle \text{II}(X,Z), \text{II}(Y,W)\rangle
    \end{align*}
    By the second Bianchi identity, we have:
    \begin{align*}
        sec_N = \lambda^2,
    \end{align*}
    where $\text{II} = \lambda g|_N$ for some constant $\lambda$. We may assume $\lambda\geq 0$.

    When $\lambda = 0$, then $N$ is the flat, totally geodesic hypersurface in $\mathbb{R}^n$, then $N$ is a subset of a flat $\mathbb{R}^{n-1}$;

    When $\lambda > 0$, then $N$ has a constant positive sectional curvature. Now consider the following.
    \begin{align*}
        c: N\to \mathbb{R}^n, x\mapsto i(x) + \frac{1}{\lambda}\vec{n},
    \end{align*}
    where $\vec{n}$ is the inner normal vector field. Then
    \begin{align*}
       \nabla_X c = \nabla_X i + \frac{1}{\lambda}\nabla_{\nabla_X i} \vec{n}  = \nabla_X i - \nabla_X i = 0 .
    \end{align*}
    Thus, we have $c$ is a constant and $|i - c| = \frac{1}{\lambda}$. Hence, the image of $i$ is a subset of the $\frac{1}{\lambda}$-ball centered on $c$. 

    If we assume $N$ to be closed, then $i(N)\simeq \mathbb{S}^{n-1}$, the round sphere.
\end{proof}

Combined with Theorem \ref{thm SY}, we have, if $(M,g)$ is an LCF manifold with non-negative scalar curvature, then the only closed totally umbilic submanifolds in the universal covers are the conformal round spheres.

\subsection{Harmonic functions on ALE manifolds}
Now, we recall the theory of harmonic functions in ALE manifolds.

\begin{theorem}[Bounded harmonic functions on ALE manifolds]\label{harmonic on ALE}
    Let $(M,g)$ be an ALE manifold. Let $u$ be a bounded function such that, outside a compact set $K$, on an ALE end it satisfies $\Delta u = 0$. Then, on this end, $\lim\limits_{x\to\infty} u(x) = c$ for some constant $c$. In fact, $u$ has an expansion at the ALE end: there exists $\epsilon>0$ such that
    \begin{align*}
        u(x) = c + A r^{2-n} +O(r^{2-n-\epsilon})
    \end{align*}
    as $r\to\infty$.
\end{theorem}
\begin{proof}
    Here, we only sketch the proof. The proof can be found in many references; see, for example, \cite{DK}.

    We first decompose $u = u_0 + u_1$ at the ALE end outside the compact set, where $u_0$ satisfies $\Delta_{\mathbb{R}^n} u_0 = 0$ and $u_1$ is $O(r^{-\epsilon})$ for some $\epsilon$. This is done by choosing a suitable weighted Sobolev space to solve the Poisson equation and using the corresponding Schauder estimate. This is possible because of the ALE assumption and the fact that the exceptional set of the Laplacian is discrete. Now, the expansion of $u_0$ follows from the Green's function expansion and the fact that the bounded global harmonic functions on $\mathbb{R}^n$ are constant. Finally, we can iteratively get better estimates using the fact that the exceptional weight is discrete.
\end{proof}

\begin{remark}
    The same techniques hold for equation $\Delta + c S$, where $c$ is some constant, since $S\in O(r^{-2-\tau})$, for an ALE end of order $\tau$.
\end{remark}

Now, we can prove the theorem about the end of the LCF, ALE manifolds.

\begin{theorem}
    For $(M,g)$, LCF, ALE manifolds of non-negative scalar curvature, if there exists $\phi$, where $\phi: M \to \mathbb{S}^n$ is a conformal embedding, then $(M,g)$ is an asymptotically Euclidean manifold, i.e., $\Gamma = \{e\}$ for all ends.
\end{theorem}
\begin{proof}
    Without loss of generality, we may assume that there is only one ALE end of order $\tau >0$ with group $\Gamma$. Our goal is to construct totally umbilic $\mathbb{S}^{n-1}/\Gamma$ in the end when $r$ is very large, where $r$ is the distance function with respect to a point $z\in M$. Due to Theorem \ref{tus in Rn}, we can conclude that $\Gamma = \{e\}$. In general, producing totally umbilic submanifolds directly in general manifolds is very difficult. However, in our case, we can use the ALE geometry at infinity and conformal equivalence to produce such submanifolds.

    Note $\phi: M \to \mathbb{S}^n$ is a conformal map. Choosing a base point $o\in M$, we can do the stereographic projection based on $\phi(o)$. The resulting map: $\bar{\phi}:M\setminus\{o\} \to \mathbb{R}^n$ is also conformal with the pull-back metric: $\bar{\phi}^* g_{\mathbb{R}^n} = u^{\frac{4}{n-2}} g$. $u$ satisfies the equation:
    \begin{align*}
        -\Delta u + a(n) S u = 0,
    \end{align*}
    where $a(n) = \frac{n-2}{4(n-1)}$ and $S = S(g)\geq 0$ is the scalar curvature of $(M,g)$.

    Since $u$ is the conformal Green's function at point $o$, away from $o$, $u$ is bounded. Thus, applying Theorem \ref{harmonic on ALE}, on the ALE end, fix a $z\in M$, and $r = d(x,z)$, we have the asymptotic on the ALE end:
    \begin{align*}
        u = Ar^{2-n} + O(r^{2-n-\epsilon}),
    \end{align*}
    for some $\epsilon >0$. The constant term is 0 since it is a compactification conformal factor. Up to a scaling, we can assume $A = 1$.

    Next, we consider the annulus region 
    \begin{align*}
        A_g(R, 4R) = \{x:x\text{ in the end}, R\leq \operatorname{dist}(x,z)\leq 4R\}.
    \end{align*}
    According to the ALE assumption, when $R$ is sufficiently large, $A_g(R,4R) \simeq [R,4R]\times\mathbb{S}^{n-1}$ is connected to the metric $C^{1,\alpha}$ close to the flat metric. 

    From our conformal map, $u^{\frac{4}{n-2}} = (1+O(r^{-\epsilon}))r^{-4}$, the metric at the end $u^{\frac{4}{n-2}}g$ can be compactified by adding a point $p$. In fact, for some small $r>0$, $(B(p,r), u^{\frac{4}{n-2}} g)\subseteq (C(\mathbb{S}^{n-1}/\Gamma), g_{flat})$, the flat cone metric. Thus, the conformal annulus $(A_g(R,4R), u^{\frac{4}{n-2}}g)$ is embedded in the flat cone $C(\mathbb{S}^{n-1}/\Gamma)$.
     Notice that  a priori this embedding of $A_g(R,4R)$ can be pretty wild, and the image may not contain any totally umbilic submanifold.
     
     Denote $A(a,b)$ as the annulus region in the flat cone $C(\mathbb{S}^{n-1}/\Gamma)$ with respect to the distance function of the vertex.

    \textbf{Claim}: for $R$ sufficiently large, there is a region $A(2R,3R)\subseteq(C(\mathbb{S}^{n-1}/\Gamma), g_{flat})$, which can be conformally embedded in the region $(A_g(R,4R), g)$.

    % Let $\rho$ be the distance function on $(C(\mathbb{S}^{n-1}/\Gamma), g_{flat})$ to the vertex.
    
    % The claim can be shown in this way: By the ALE assumption, the rescaled sequence:

    % \begin{align*}
    %     \begin{tikzcd}[ampersand replacement=\&]
    %      \& (A_g(R, 4R) , R^{-2} g) \arrow[r, "C^\infty"] 
    %      \& (A(1,4), g_{flat})
    %   \end{tikzcd}
    % \end{align*}
    % as $R\to \infty$. Then, for $R$ large enough, $(A_g(R, 4R), g)$ is $R\delta(R)$-close to $(A(R,4R), g_{flat})$, where $\delta(R)\to 0$ as $R\to \infty$.

    % In addition, By the asymptotic of $u$, we have $(A_g(R, 4R), u^{\frac{4}{n-2}} g)$ is $R^{-1}\epsilon(R)$-close to $(A(R,4R), \rho^{-4}g_{flat})\cong (A((4R)^{-1}, R^{-1}), g_{flat})$, for some $\epsilon(R)\to 0$ as $R\to \infty$. 

    % Since $A(2R,3R)\subseteq A(R,4R)$, and the Hausdorff distance:
    % \begin{align*}
    %     d_\mathcal{H} ((A(R,4R), \rho^{-4}g_{flat}), (A(2R, 3R), \rho^{-4}g_{flat})) \geq R^{-1}
    % \end{align*}

    % Thus, for $R$ large enough, by looking at the $R^{-1}\epsilon(R)$ neighborhood, we prove the claim.

    \begin{proof}[\textbf{Proof of the Claim}]: Consider the conformal metric, $\hat{g} = r^{-4} g$, with the inverted coordinate: $x' = \frac{x}{r^2}$, where $r$ is the original distance function. Note, the distance to $p$, $\rho = |x'|$, is $r^{-1}$. In particular, $A_g(a,b) = A_{\hat{g}}(b^{-1}, a^{-1})$. 

        Since we consider $R$ large, it is enough to consider $((B(p,\delta), \hat{g})$, for some $\delta>0$ sufficiently small. Consider the flat conformal metric, 
        \begin{align*}
            g_0 = u^{\frac{4}{n-2}} g = (1+O(\rho^\epsilon))\hat{g}.
        \end{align*}
        Note, $(B(p,\delta), g_0)\subseteq ((C(\mathbb{S}^{n-1}/\Gamma), g_{flat})$, is part of the flat cone, with the vertex $p$. Then, we can compare two distance functions: for $\delta$ small, on $((B(p,\delta), \hat{g})$, we have the function $q(x) = \frac{\operatorname{dist}_{g_0}(x,p)}{\operatorname{dist_{\hat{g}}(x,p)}}$, which is continuous away from $p$, and $\lim\limits_{x\to p}q(x)=1$. Thus, there exists a function $\eta$, with $\lim\limits_{t\to0}\eta(t) = 0$, such that:
        \begin{align*}
            (1-\eta(\delta))\operatorname{dist}_{\hat{g}}(x,p)\leq\operatorname{dist}_{g_0}(x,p)\leq(1+\eta(\delta))\operatorname{dist}_{\hat{g}}(x,p),
        \end{align*}
        for any $x\in B(p,\delta)$. If $x\in A(2R,3R)$, then $\operatorname{dist}_{\hat{g}}(x,p) \leq (2R)^{-1}(1+\eta(\delta)) \leq R^{-1}$, and similarly, $\operatorname{dist}_{\hat{g}}(x,p) \geq (4R)^{-1}$. Since all of the metrics above are conformal, if we choose $\delta$ small enough, then the claim is proved.

        \end{proof}

        As shown in Figure \ref{fig:psi-map}, there exists a conformal map, $\psi$, from the ALE end $E$, to the flat cone. This $\psi$ is obtained by composing with an additional inversion map. Thus, we can embed the region $A_g(R,4R)$ in the flat cone. From the Claim, it will be close to $A(R,4R)$.

    % From our conformal map, $u^{\frac{4}{n-2}} = (1+O(r^{-\epsilon}))r^{-4}$, thus, the metric on the end $u^{\frac{4}{n-2}}g$ can be compactified by adding a point, $p$. On the other hand, $u^{\frac{4}{n-2}}g$ is a flat metric. Thus, consider function 
    % \begin{align*}
    %     \rho(x) = dist_{Euc}(x,p)
    % \end{align*}
    % which will be the Euclidean distance from $x$ to the infinity $p$. It is easy to see $\rho(x) = O(r^{-1})$. 

    % Now, consider the new conformal change: 
    % \begin{align*}
    %     v^{\frac{4}{n-2}} g = u^{\frac{4}{n-2}}\rho^{-4}(x)g
    % \end{align*}
    % The new metric, by construction, is also flat. On the other hand, up to a scaling, $v^{\frac{4}{n-2}} = (1+O(r^{-\epsilon}))$.

    % By the assumption of ALE, for any $\delta>0$, we can choose $R$ large enough, $A_g(R, 4R)$ is $\delta$-$C^{k,\alpha}$ close to $A(R,4R)$. Thus, $(A_g(R,4R), v^{\frac{4}{n-2}}g)$ is $\delta'$-$C^{k,\alpha}$ close to $(A(R,4R), g_{flat})$. Thus, for $R$ large enough, we have $A(2R,3R)\subset (A_g(R,4R), v^{\frac{4}{n-2}}g)$ via the almost isometry diffeomorphism. Hence we prove the claim.

    On the flat metric of $A(2R,3R)\cong [2R,3R]\times \mathbb{S}^{n-1}/\Gamma$, write $([2R,3R]\times \mathbb{S}^{n-1}/\Gamma,dt^2+t^2 g_{\mathbb{S}^{n-1}/\Gamma})$.
  Then for the slice $\{\frac{5}{2}R\}\times\mathbb{S}^{n-1}/\Gamma\subseteq([R,3R]\times\mathbb{S}^{n-1}/\Gamma,\delta)$, the second fundamental form:
  \begin{align*}
        \text{II} = \mathcal{L}_{\partial t}(dt^2 + t^2 g_{\mathbb{S}^{n-1}/\Gamma}) = 2t g_{\mathbb{S}^{n-1}/\Gamma} = \frac{2}{t} g|_{\{\frac{5}{2}R\}\times\mathbb{S}^{n-1}/\Gamma}.
  \end{align*}
  
  Thus, $\{\frac{5}{2}R\}\times\mathbb{S}^{n-1}/\Gamma\cong\mathbb{S}^{n-1}/\Gamma$ is a totally umbilic submanifold. For $R$ large, $\{\frac{5}{2}R\}\times\mathbb{S}^{n-1}/\Gamma\subseteq(A(2R,3R),\delta)$, as a totally umbilic submanifold.
  
  Since the totally umbilic submanifold is invariant under conformal mapping, there exists $N\subset(M,g)$, $(N,g|_N)$ is $C^{1,\alpha}$ close to $(\{\frac{5}{2}R\}\times\mathbb{S}^{n-1}/\Gamma,g_{\mathbb{S}^{n-1}/\Gamma})$ and $N\cong\mathbb{S}^{n-1}/\Gamma_i$, which is totally umbilic. This violates Theorem \ref{thm SY}
  and Theorem \ref{tus in Rn} unless $\Gamma = \{e\}$. Thus, this end is an AE end. 

  We can repeat the process and conclude that $(M,g)$ contains only AE ends.

  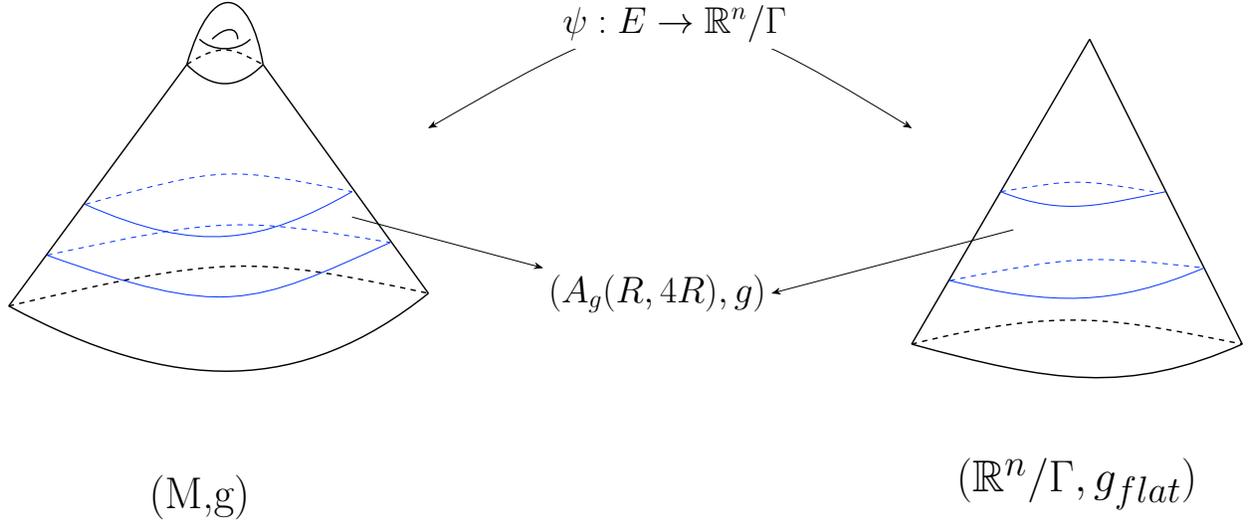
\begin{figure}[!ht]
     \centering
     \resizebox{1\textwidth}{!}{\begin{circuitikz}
\tikzstyle{every node}=[font=\LARGE]
%\draw[fill={rgb,255:red,255; green,255; blue,255}] (1,16.5) rectangle (27.25,5.5);
\draw [line width=0.8pt, short] (5.5,15) -- (2,10.25);
\draw [line width=0.8pt, short] (7,15) -- (10.25,10.5);
\draw [line width=0.8pt, short] (2,10.25) .. controls (5.25,8.5) and (7.75,8.5) .. (10.25,10.5);
\draw [line width=0.8pt, dashed] (2,10.25) .. controls (6.25,11.25) and (6.75,11.25) .. (10.25,10.5);
\draw [line width=0.8pt, short] (5.5,15) .. controls (6,14.5) and (6.5,14.5) .. (7,15);
\draw [line width=0.8pt, dashed] (5.5,15) .. controls (6.25,15.5) and (6.5,15.25) .. (7,15);
\draw [line width=0.8pt, short] (5.5,15) .. controls (6,16.75) and (6.75,16.5) .. (7,15);
\draw [line width=0.8pt, short] (5.75,15.5) .. controls (6,15.25) and (6.5,15.25) .. (6.75,15.5);
\draw [line width=0.8pt, short] (6,15.5) .. controls (6.25,15.75) and (6.5,15.75) .. (6.5,15.5);
\draw [ color={rgb,255:red,4; green,51; blue,255}, short] (3.5,12.25) .. controls (5.75,11.25) and (7,11.5) .. (8.75,12.5);
\draw [ color={rgb,255:red,4; green,51; blue,255}, short] (2.75,11.25) .. controls (6.25,10) and (6.75,10.25) .. (9.5,11.5);
\draw [ color={rgb,255:red,4; green,51; blue,255}, dashed] (3.5,12.25) .. controls (6.25,13) and (6.25,13) .. (8.75,12.5);
\draw [ color={rgb,255:red,4; green,51; blue,255}, dashed] (2.75,11.25) .. controls (6.25,12) and (6.5,12) .. (9.5,11.5);
\node [font=\LARGE] at (14.75,10.5) {$(A_g(R,4R), g)$};
\draw [<->, >=Stealth] 
  (10.25,13.75) .. controls (15,16.5) and (15.25,16.5) .. (19.75,13.75)
  node[pos=0.5, fill=white]{$\psi: E \rightarrow \mathbb{R}^n/\Gamma$};
\draw [line width=0.8pt, short] (23.25,15.5) -- (19.75,9.5);
\draw [line width=0.8pt, short] (23.25,15.5) -- (26.25,9.5);
\draw [line width=0.8pt, short] (19.75,9.5) .. controls (22.5,8.75) and (24,8.5) .. (26.25,9.5);
\draw [line width=0.8pt, dashed] (19.75,9.5) .. controls (22.75,10.25) and (23.75,10) .. (26.25,9.5);
\draw [ color={rgb,255:red,4; green,51; blue,255}, short] (21.5,12.5) .. controls (22.75,12) and (23.5,12.25) .. (24.75,12.5);
\draw [ color={rgb,255:red,4; green,51; blue,255}, short] (20.5,10.75) .. controls (22.5,10.25) and (23.75,10.25) .. (25.5,11);
\draw [ color={rgb,255:red,4; green,51; blue,255}, dashed] (21.5,12.5) .. controls (23,12.75) and (23,12.75) .. (24.5,12.5);
\draw [ color={rgb,255:red,4; green,51; blue,255}, dashed] (20.5,10.75) .. controls (23,11.25) and (23.25,11.25) .. (25.5,11);
\draw [->, >=Stealth] (8.75,12) -- (12.5,11);
\draw [->, >=Stealth] (21.75,11.75) -- (17,10.5);
\node [font=\Huge] at (5.75,6.5) {(M,g)};
\node [font=\Huge] at (23,6.75) {$(\mathbb{R}^n/\Gamma, g_{flat})$};
\end{circuitikz}}
     \caption{Conformal mapping $\psi : E \to \mathbb{R}^n/\Gamma$, from the ALE end to the flat cone.}
     \label{fig:psi-map}
  \end{figure}
  \end{proof}

\subsection{The local isotropy group and developbility}\label{isotropy ALE section}
In addition, we can prove the following structure theorem for the universal cover of the LCF, ALE manifolds.
\begin{proposition}\label{ALE structure}
    Let $(M,g)$ be an LCF, ALE manifold with non-negative scalar curvature, then its universal cover $\tilde{M}$ is diffeomorphic to $\mathbb{S}^n\setminus\Lambda\cup I$, where $\Lambda$ is the limit set of $\pi_1(M)$ under the conformal diffeomorphism, and $I$ is a discrete set of points that correspond to the AE ends.
\end{proposition}
\begin{proof}
    Now, $(M,g)$ is an ALE manifold. By the previous argument, since there is a conformal developing map $\phi:\tilde{M}\to \mathbb{S}^n$, the only ALE ends in $\tilde{M}$ are AE ends, possibly infinitely (countably) many of them.

    One way to see that each AE end can be compactified by adding a point is to use the conformal embedding $\bar{\phi}:\tilde{M}\to\mathbb{R}^n$ and to use the expansion of the harmonic function on the end of the AE. Note that in this case, our group $\Gamma$ is trivial, thus, we can compactify it by adding a manifold point.

    Another way to see this is to use the exhaustion of totally umbilic manifolds on the end: As before, consider the $i$-th end. For $r$ large enough, instead of $A$, we consider a family of annuli:
  \begin{align*}
      A_k=\{x:x\text{ in the $i$-th end }, kr\leq \operatorname{dist}(x,z)\leq (k+1)r\}.
  \end{align*}
  
  Repeat the above argument, $\exists N_k\subset A_k$, totally umbilic, and $N_k\cong\mathbb{S}^{n-1}$. Moreover, let $C_k\subset \tilde{M}$ be a compact subset such that $\partial C_k=N_k$. Then we have $C_1\subset C_2\subset C_3...\subset C_k\subset C_{k+1}\subset...$. Moreover, $\tilde{M}\supset \cup_{k=1}^\infty C_k$, $\{C_k\}$ is an exhaustion of the $i$-th end. 
  
  Now, under the conformal map $\phi:\tilde{M}\to\mathbb{S}^n$, $\phi(N_k)=S_k\cong\mathbb{S}^{n-1}$, $S_i\bigcap S_j=\emptyset$ for $i\neq j$, $\phi(C_k)\subset\phi(C_{k+1})$, we see that $\operatorname{rad}(S_k)$ is decreasing. By Theorem 2.7 in \cite{SYconformal}, $\lim\limits_{k\to\infty}\operatorname{rad}(S_k)=0$, otherwise, the Hausdorff dimension of $\mathbb{S}^n\setminus\phi(M)$ is $n$. Thus, $\lim\limits_{k\to\infty}S_k=\{x_i\}$. Thus, we have a one-point compactification for each end.

  Let the compactified universal cover
  \begin{align*}
      \overline{M}=\phi(\tilde{M})\cup_i \{x_i\},
  \end{align*}
  with the metric to be the spherical metric $h$. By Van Kampen's theorem and $n\geq 3$, $\overline{M}$ is simply connected. The deck transformation can be extended by the effective action on $\overline{M}$ by maps $x_i$ to some $x_j$ with the only possible fixed point being $\{x_i\}$. Then $\overline{M}\cong\mathbb{S}^n\setminus\Lambda$, where $\Lambda$ is the limit set of deck transformations on $\tilde{M}$. This is because the local isometry group is finite. 
  
  For each $x_i$, if $x_i\in \Lambda$, it is an isolated limit point. Then $\exists \gamma\in\pi_1(M)=\operatorname{Deck}(\overline{M})$ and $x\in \phi(\tilde{M})$ such that $\lim\limits_{n\to\infty}\gamma^n(x)=x_i$. Then $x_i$ will not correspond to an AE end, which is a contradiction. Thus $x_i\notin\Lambda$. Thus, $\tilde{M}\cong \mathbb{S}^n\setminus\Lambda\cup I$, where $I$ is a discrete set of points.
\end{proof}

Now we can prove Theorem \ref{group injective}. We can show that for each $\Gamma_k$, there is a group homomorphism $i_k:\Gamma_k\to\pi_1(M)$, which is an injective map.

\begin{proof}[Proof of Theorem \ref{group injective}]
    Consider 
    \begin{align*}
      S_r=\{x:x\text{ in the $k$-th end },\operatorname{dist}(x,z)=r\}  
    \end{align*}
    to be the distance sphere. Then, for $r$ large enough, $(S_r,j^*g)$ and $(\mathbb{S}^{n-1}/\Gamma_k,g_{\mathbb{S}^n/\Gamma_k})$ are $C^{1,\alpha}$ close, and $S_r\cong \mathbb{S}^{n-1}/\Gamma_k$. ($j^*g$ is the pullback metric under $j:S_r\to M$.)
  
  Let $(\tilde{M},p^*g)$ be the Riemannian cover with the covering map: $p:\tilde{M}\to M$. Then choose $z_0\in p^{-1}(z)$, there is a lift of $S_r$, $\tilde{S_r}$, by lifting the geodesic. Then $\tilde{S_r}$ is in an AE end. For $r$ large enough, $\tilde{S_r}\cong \mathbb{S}^n$. Thus, we have the map $p:\tilde{S_r}\to S_r$, which is a covering map. Since $\pi_1(\tilde{S_r})=\{e\}$, then $\tilde{S_r}$ is the universal cover of $S_r$.
  
  Now, let 
  \begin{align*}
      j_k:\mathbb{R}^+\times\mathbb{S}^{n-1}/\Gamma_k\to M
  \end{align*}
  be the embedding. This induces the homomorphism
  \begin{align*}
      i_k:\pi_1(\mathbb{R}^+\times\mathbb{S}^{n-1}/\Gamma_k)\cong \Gamma_k\to \pi_1(M).
  \end{align*}
  
  Let $\gamma\subset \mathbb{R}^+\times\mathbb{S}^{n-1}/\Gamma_k$, a closed loop such that $[\gamma]\in\operatorname{ker}(i_k)\lhd\pi_1(\mathbb{R}^+\times\mathbb{S}^{n-1}/\Gamma_k)$. Note that $\gamma$ is homotopically equivalent to $\gamma\in S_r$, with $[\gamma]\in \pi_1(S_r)\cong \Gamma_k$. Since $i_k([\gamma])$ is trivial, then $\tilde{\gamma}\subset \tilde{S_r}$, the lift of $\gamma$, is a closed loop, which represents the trivial deck transformation of $\pi_1(M)$. Since $\pi_1(\tilde{S_r})=\{e\}$, then $[\gamma]$ represents the trivial deck transformation in $\pi_1(S_r)$, thus, $[\gamma]=e\in\pi_1(S_r)=\Gamma_k$. Thus, $i_k$ is injective. 
\end{proof}
In particular, for each $\gamma\in\Gamma_k< \pi_1(M)$, as a deck transformation, it fixes $\{x_i\}$. And it induces a map on a small ball $B_\epsilon$ centered on $x_i$. Then from the Liouville theorem, we immediately have:

\begin{corollary}\label{cor 1}
  For each $\Gamma_k<\pi_1(M)$, $\rho(\Gamma_k)<C(n)$ which fixes $\{x_k\}$, where $C(n)$ is the conformal group of $\mathbb{S}^n$.
\end{corollary}

Note that the action $\rho(\pi_1(M))$ extends to the compactification $\overline{M}$. Moreover, this action is properly discontinuous. Thus $M':=\overline{M}/\rho(\pi_1(M))$, is a compact orbifold (\cite{Thurston_orbifold}).
In particular, 
\begin{align*}
    M'=M\cup\{x_1,...,x_m\},
\end{align*}
which is the compactification of $M$. Since it is a quotient of a manifold $\bar{M}$, $M'$ is a good orbifold, with each local group $\Gamma_i$ injecting into $\pi_1(M)$. 

Note that we can also define a \textbf{Kleinian orbifold} (See Section \ref{s:lcf and Kleinian} for Kleinian manifolds).
Thus, $M'$ is a Kleinian orbifold.

\begin{proposition}\label{manifold cover}
    Let $M'$ be a good Kleinian orbifold. Then there exists a compact Kleinian manifold $N$, such that:
    \begin{align*}
        p: N\to M'
    \end{align*}
    is a finite covering map.
\end{proposition}
\begin{proof}
    Since $M = \Omega/\Gamma$, for some open set $\Omega\subseteq \mathbb{S}^n$, and $\Gamma\leq C(n)$, by Theorem \ref{theorem.selberg}, there exists a normal subgroup $H\leq \Gamma$ of finite-index and torsion-free. Note that the torsion-free action in $C(n)$ is fixed-point-free (Remark \ref{elliptic}). Thus, the resulting normal cover space $N\cong \Omega/H$ a compact manifold, whose fundamental group is isomorphic to $H$, and the deck transformation group is isomorphic to $\Gamma/H$.
\end{proof}

\subsection{The connection with the orbifolds}\label{good orbifolds proof}

Now, we link the geometry of LCF orbifolds with positive scalar curvature with  ALE manifolds of nonnegative scalar curvature. In the following, we show that the conformal compactification of the nonnegtative scalar curvature, LCF, ALE manifolds is a compact LCF orbifold with positive Yamabe invariant. 

\begin{theorem}[Conformal compactification with positive Yamabe invariant]\label{positive yamabe}If $(M,g)$ is a LCF,
nonnegative scalar curvature ALE manifold. 
The compactified ALE manifold $M'$ has positive Yamabe invariant, that is, $Y_{orb}(M')>0$.
    
\end{theorem}
\begin{proof}
    Without loss of generality, we assume that $(M,g)$ has only one ALE end. First, we need a well-defined orbifold metric. Fix $z\in M$, let $r(x)=\operatorname{dist}(x,z)$ be the distance function of $z$. Consider the metric:
\begin{align*}
    \hat{g}=r^{-4}g.
\end{align*}
In the inverted coordinate, let $\rho=\frac{1}{r}$, then by the ALE assumption, on the ball $B(x_0,\delta):=\{x:\rho(x)\leq \delta\}$, where $x_0$ is an orbifold point and $\delta$ is sufficiently small, we have 
\begin{align*}
    \hat{g}=\delta+O(\rho^{\tau}).
\end{align*}
Note that $B(x_0,\delta)\cong B_{\mathbb{R}^n}(0,1)/\Gamma$ is an orbifold with an isolated orbifold point $0$.

Second, we want a removable singularity argument.  We take the universal cover $P: B_{\mathbb{R}^n}(0,\delta) \setminus \{0\}\to B(x_0,\delta)\setminus\{x_0\}$, and the $\Gamma$-invariant, pull-back metric $(B_{\mathbb{R}^n}(0,\delta)\setminus\{0\}, P^* \hat{g})$. There are many ways to prove this. We will prove the removable singularity theorem using injection of the local group into $\pi_1(M)$, and conformal embedding. 
Consider the universal cover: $\pi: \Tilde{M}\to M$. Since $\Gamma$ injects into $\pi_1(M)$, there exists a conformal diffeomorphism $\Phi:B_{\mathbb{R}^n}(0,\delta)\setminus\{0\}\to B_{\mathbb{R}^n}(0,\delta)\setminus\{0\}$ such that $\Phi^* P^* \hat{g}=u^{\frac{4}{n-2}}g_{\mathbb{R}^n}$, hence it extends smoothly at the origin by a smooth function $u\in C^\infty(B_{\mathbb{R}^n}(0,\delta))$ since $u$ satisfies the conformal equation $-\Delta u + S(\hat{g}) u = 0$. Thus, it follows from the removability of the singularity of $u$. We denote the compactified metric $(M',g')$.

By lemma 3.4 of \cite{AB04}, on the orbifold $(M',g')$ there exists a conformal factor such that the conformal metric $g''\in[g']$ has the scalar curvature $S(g'')$ which does not change sign. To get the ALE metric $g$ on $M$, it is equivalent to finding a conformal factor $v$ such that:
\begin{align*}
    L'' v=-\Delta_{g''} v+ a(n) S(g'')v=S(g) v^{\frac{n+2}{n-2}}\geq 0
\end{align*}
on $M\subset M'$. Since $v>0$, $M'$ is compact, $v(x)\to +\infty$ as $x\to x_\infty$, there exists $x_0\in M$ such $v(x_0)=\min_{x\in M'} v(x)=\min_{x\in M} v(x)$. At that point, $-\Delta_{g''}v(x_0)\leq0$, and thus $S(g'')(x_0)\geq0$. Since $S(g'')$ does not change sign, by strong maximal principle, $S(g'')>0$. Thus, the Yamabe invariant $Y_{orb}(M')>0$. 

\end{proof}

As a corollary, for every such LCF, ALE manifold, we have that such an ALE metric is equivalent to the conformal Green's function metric on a compact LCF orbifold with positive Yamabe invariant. The conformal factors are the superposition of conformal Green's functions by Bôcher's theorem. Thus, we have the optimal decay rate.

\begin{corollary}\label{order of decay}
    Let $(M,g)$ be an LCF, scalar-flat, ALE manifold of order $\tau$, for any $\tau >0$, then there exist charts at infinity for every end such that:
    \begin{align*}
    &g=\delta+O(r^{-n+2})\\
    &\partial^m g=O(r^{-n+2-|m|}).
    \end{align*}
    i.e. they are ALE of order $n-2$. If $(M,g)$ is an LCF, nonnegative scalar curvature, ALE manifold, then there exists a conformal metric such that it is scalar-flat, ALE manifold of order $n-2$.
\end{corollary}
\begin{proof} Note that the ALE manifold in this case is scalar-flat. It is known that on the orbifold of positive Yamabe invariant, there exists positive conformal Green's functions. Using Bôcher's theorem, the conformal factors are the superposition of conformal Green's functions. In the conformal normal coordinates, the conformal Green's function has local expansion:
\begin{align*}
    G = r^{2-n} + A + O(r),
\end{align*}
which is without log terms, by the LCF assumption (See Lemma 6.4 of \cite{lee1987yamabe}). The optimal decay rate follows from direct computation.
\end{proof}
\begin{remark}\label{remark of order of decay}
    The same argument works for LCF, nonnegative scalar curvature, ALE orbifolds as well. To relate the manifolds and orbifolds with at most isolated singularities, we need the conformal blow-up at the singularities to produce ALE ends. This will be proved in Lemma \ref{ALE orbifold blow-up}.
\end{remark}
This is a much easier proof of the optimal ALE order for locally conformally flat scalar-flat ALE manifolds. The optimal decay rate for obstruction-flat scalar-flat ALE manifolds is proved by \cite{AcheViaclovsky}.

Conversely, we can reverse the above picture, starting from an LCF orbifold of positive orbifold Yamabe invariant with at most isolated singularities to construct the ALE manifolds. This will prove our good orbifold theorem for all LCF orbifolds with positive scalar curvature (Theorem \ref{good orbifold}).

\begin{proof}[Proof of Theorem \ref{good orbifold}]
    Let $(M,g)$ be our orbifold, with isolated singularities $\{x_i\}$. By the assumption of positive scalar curvature, there exists a positive Green's function of the conformal Laplacian blow-up at the orbifold point $x_i$. The superposition of all such Green's functions, $G$, will give us the conformal factor such that $(M\setminus\{x_i\}, G^{\frac{4}{n-2}}g)$ is a multi-ALE manifold of order $n-2$. Note $\pi_1^{orb}(M) \cong \pi_1(M\setminus\{x_i\})$.

    From Theorem \ref{ALE structure} and Corollary \ref{cor 1}, each isotropy group $\Gamma_{x_i}$ injects into the fundamental group $\pi_1(M\setminus\{x_i\})$. Thus, by a well-known theorem of orbifolds, an orbifold is good if and only if each isotropy group injects into the orbifold fundamental group (follow the same spirit of Theorem \ref{group injective}). Hence done. 
\end{proof}
\begin{remark}
    We use the ALE structure (Theorem \ref{ALE structure} and Corollary \ref{cor 1}) to prove the good orbifold theorem (Theorem \ref{good orbifold}). Note that the proof of Theorem \ref{ALE structure} uses the local structure of flat cone and the LCF condition as well. So the order of the proof does not matter, if we can show an orbifold version of Theorem \ref{thm SY}. Starting with the ALE manifolds seems more natural to the author.
\end{remark}

%%%%%%%%%%%%%%%%%%%%%%%%%%%%%%%%%%%%%%%%%%%%%%%%%%%%%%%
\section{Some classification results}

\subsection{Relation with the conformal group}
From the previous discussion, the compactified nonnegative scalar curvature ALE space, $M'$, can be viewed as the quotient of the compactified universal cover $\overline{M}$ by $\rho(\pi_1(M))$, where $\rho(\pi_1(M))$ acts on $\overline{M}$ properly discontinuously, with isolated fixed points, and the corresponding isotropy group is finite. To understand this quotient, we need to understand the subgroup of the conformal group of $\mathbb{S}^n$.

To study the local isotropy subgroups, we first show that we can conjugate the local group $\Gamma_i$ to be a subgroup of ${\rm{O}}(n)$.
\begin{lemma}\label{local group}
    Let $\Gamma\leq C(n)$ be a local group corresponding to the orbifold point $x$, then $\exists \gamma\in C(n)$ such that $\gamma^{-1}\Gamma\gamma\in {\rm{O}}(n)$.
\end{lemma}
\begin{proof}
    Let $g\in \Gamma$. Since $g$ fixes $x$, $g$ is of finite order, then $g$ is elliptic. Thus, $\Gamma$ is an elliptic subgroup. Note that we can always consider the chart where $\Gamma$ acts as the finite subgroup of the linear group fixing the origin (linear chart). Thus, $\Gamma$ will fix some $K\cong \mathbb{S}^{n-1}$ in the local universal cover.
    
    Now, there exists $\gamma\in C(n)$ such that $\gamma^{-1}g\gamma\in {\rm{O}}(n+1)$ (\cite{conformalbook}, Chapter 2). Since $g$ fixes $x$, then $\gamma^{-1}g\gamma\in {\rm{O}}(n)$. By ${\rm{O}}(n+1)$ acting transitively on $\mathbb{S}^n$, we can also compose $\gamma$ with elements in ${\rm{O}}(n+1)$ such that $\gamma$ fixes $x$. Hence $\gamma^{-1}g\gamma$ also fixes $-x$, the antipodal point, and we can write $g\in C(\mathbb{R}^n)$, the conformal group of $\mathbb{R}^n$, with the following expression:
    \begin{align}\label{g expression}
            \gamma^{-1}g\gamma(v)=A(v)
    \end{align}
    for $A\in {\rm{O}}(n)$, for some $v\in \mathbb{R}^n$. (We let $x$ be $\infty$ and $-x$ be $0$.) 
    
    Now, $\forall f\in\Gamma$, $\gamma^{-1} f\gamma$ also fixes $x$. Thus, we write $\gamma^{-1}f\gamma(v)=cB(v)+a$. $c$ is a constant, $B\in {\rm{O}}(n)$, and $a\in\mathbb{R}^n$. Since $f$ is of finite order, then $c=1$. $\gamma^{-1}f\gamma(v)=B(v+b)-b$.

    Since $f$ and $g$ map some $K\cong \mathbb{S}^{n-1}$ to themselves, then $\gamma^{-1}f\gamma$ and $\gamma^{-1}g\gamma$ map $\gamma^{-1}(K)\cong\mathbb{S}^{n-1}$ to itself. Now, from (\ref{g expression}), the only fixed $\mathbb{S}^{n-1}$ are
    \begin{align*}
            S_R = \{v\in\mathbb{R}^n:|v|=R\}.
    \end{align*}
    Then $\gamma^{-1}f\gamma$ fixes $S_R$ if and only if $b=0$. Thus, $\gamma^{-1}f\gamma\in {\rm{O}}(n)$. Hence $\gamma^{-1}\Gamma\gamma\leq {\rm{O}}(n)$. In particular, $\gamma^{-1}\Gamma\gamma$ fixes $x$, $-x$.
\end{proof}

Thus, for a subgroup $\Gamma\leq {\rm{O}}(n)$ acting on $\mathbb{S}^n$, if it has a fixed point, it fixes the antipodal point as well. It seems that the non-trivial orbifold points must appear in pairs. This will be the case if this subgroup does not contain parabolic elements. The proof of the following theorem can be found in Appendix \ref{orbifold points in the quotient}.
\begin{theorem}\label{orbifold points num}
        Let $G\leq C(n)$, a discrete subgroup acting on $\mathbb{S}^n$ properly discontinuously. Denote the limiting set of $G$ as $\Lambda$. Assume there are no parabolic elements in $G$ and the Hausdorff dimension $\operatorname{dim}_{\mathcal{H}}(\Lambda) < \frac{n-2}{2}$, then the fixed points of the local isotropy subgroup $\Gamma$ must appear in pairs unless the quotient is non-orientable.
\end{theorem}

In particular, the above theorem applies when the fundamental group is a Schottky group. We have the following corollary regarding the number of orbifold points and the corresponding non-trivial ALE ends.

\begin{corollary}\label{orientable ALE end} Let $(M,g)$ be an ALE LCF manifold. If $M$ is orientable and there is no parabolic element in $\pi_1(M)$, then the ALE ends with nontrivial group must appear in pairs.
\end{corollary}
\begin{proof}
Let $M'$ be the compactified ALE manifold. 
        By Theorem \ref{positive yamabe}, we have $Y_{orb}(M') > 0$. Thus, the compact manifold cover $N$, as in Proposition \ref{manifold cover}, has $Y(N)>0$. By Proposition 4.7 in \cite{SYconformal}, $\operatorname{dim}_{\mathcal{H}}(\Lambda)\leq \frac{n-2}{2}$. Since $(N,g')$ has a positive Yamabe constant, by Corollary 3.4 in \cite{Nayatani}, $\operatorname{dim}_{\mathcal{H}}(\Lambda)<\frac{n-2}{2}$. Thus, we can apply Theorem \ref{orbifold points num}.
\end{proof}
Corollary \ref{orientable ALE end} will be used to prove the ALE end structure part in Theorem \ref{classification}.

\subsection{Low dimension classifications and the proof of Theorem \ref{classification}} \label{classification proof}

From Theorem \ref{positive yamabe}, there is a one-to-one correspondence between LCF orbifolds with positive Yamabe invariant and LCF, scalar-flat ALE spaces. We obtain some classification theorems of the orbifolds in low dimensions. Thus, we also classify the ALE spaces in these cases.

Specifically, when $n\leq 4$, we have a classification of the LCF orbifolds with positive scalar curvature with the help of Ricci flow.

\begin{theorem}[3D, \cite{GL},\cite{Ilimit},\cite{KL}]\label{3d} The 3-D LCF orbifolds with positive scalar curvature are diffeomorphic to the connected sum of the quotients $\mathbb{S}^3/\Gamma$ and $\mathbb{S}^1\times \mathbb{S}^2/\Gamma$, i.e.,
\begin{align*}
    M\cong \mathbb{S}^3/\Gamma_1\#\mathbb{S}^3/\Gamma_2\#...\mathbb{S}^3/\Gamma_k\#\mathbb{S}^2\times \mathbb{S}^1/\Gamma'_1\#...\mathbb{S}^2\times \mathbb{S}^1/\Gamma'_m,
\end{align*}
where all the orbifold points are $\mathbb{Z}_2$-quotient singularities.
\end{theorem}

\begin{theorem}[4D, \cite{Ilimit}, \cite{Ha}, \cite{ChenZhu}]\label{4d}
    The 4-D LCF orbifolds with positive scalar curvature are diffeomorphic to the connected sum of the quotients $\mathbb{S}^4/\Gamma$ and $\mathbb{S}^1\times \mathbb{S}^3/\Gamma$, i.e.,
    \begin{align*}
    M\cong \mathbb{S}^4/\Gamma_1\#\mathbb{S}^4/\Gamma_2\#...\mathbb{S}^4/\Gamma_k\#\mathbb{S}^3\times \mathbb{S}^1/\Gamma'_1\#...\mathbb{S}^3\times \mathbb{S}^1/\Gamma'_m.
\end{align*}
\end{theorem}

In the compactified universal cover, $\overline{M}$, of the ALE manifold, by positive Yamabe, Proposition 4.7 in \cite{SYconformal}, and Corollary 3.4 \cite{Nayatani}, the Hausdorff dimension of the limiting set induced by $\rho(\pi_1(M))$, $\Lambda$, is strictly less than $\frac{n-2}{2}$. In particular, when $n = 3, 4$, the non-negative scalar curvature implies $\dim_{\mathcal{H}}(\Lambda)<1$.

\begin{proof}[Proof of Theorem \ref{classification}]
    Suppose $\dim_{\mathcal{H}}(\Lambda)<1$, then, we can apply Theorem 6.2 in \cite{Ilimit} and Selberg's lemma (Theorem \ref{theorem.selberg}). There is a torsion-free normal subgroup of $\rho(\pi_1(M))$ containing no parabolic elements, and there is a finite cover of $\overline{M}/\rho(\pi_1(M))$ (possibly a finite cover of the manifold cover) that is diffeomorphic to
    \begin{align*}
    \mathbb{S}^1\times \mathbb{S}^{n-1}\# ...\# \mathbb{S}^1\times \mathbb{S}^{n-1} = k(\mathbb{S}^1\times \mathbb{S}^{n-1}),
    \end{align*}
    for some $k>0$.

    The low-dimensional topological classification is given by Theorem \ref{3d}, \ref{4d}.
    For the conformal connected sum: note the above decomposition can be viewed as the connected sum of orbifolds near the manifold points. Thus, we can apply Theorem 2.7 from \cite{Idecomposition}.
    
    For the last statement in Theorem \ref{classification}: If we in addition assume $(M,g)$ is orientable, since $\dim_{\mathcal{H}}(\Lambda)<1$, there is no parabolic element. Then we can apply Corollary \ref{orientable ALE end}. Note, when $n$ is odd, the local isotropy group is isomorphic to $\mathbb{Z}_2\subseteq {\rm{O}}(n)$, which is orientation-reversing. So by Theorem~\ref{group injective}, when $M$ is orientable, then it has no such orbifold points.  
\end{proof}

Note in \cite{Ilimit}, the 4-D classification is actually obtained by Ricci flow on the \textbf{Positive Isotropic Curvature (PIC)}. Thus, we can derive the corollary of the structure of 4-D PIC orbifolds.

\begin{proof}[Proof of Corollary \ref{isotropy}]
By \cite{ChenZhu}, a 4-D orbifold admits a positive isotropic curvature metric if and only if it admits an LCF metric with positive scalar curvature. Since in 4-D, nonnegative scalar curvature implies $\dim_{\mathcal{H}}(\Lambda)<1$, the corollary follows from Theorem \ref{classification}.
\end{proof}

In particular, the non-trivial ALE ends of an orientable ALE space appear in pairs. 
An example can be constructed as follows: on $\mathbb{S}^n/\Gamma$, where $\Gamma\leq {\rm{O}}(n)$ has two fixed points, $s$ and $-s$, with the standard spherical metric. Consider the Green's function metric blown up at $s$ and $-s$, the resulting manifold $(M, G^{\frac{4}{n-2}}g_{\mathbb{S}^n})$ is diffeomorphic to the standard Schwarzschild metric modulo $\Gamma$. We call such a manifold the \textbf{Schwarzschild ALE} manifold. 

In the non-orientable case, we can easily construct an LCF manifold with one end:

\begin{example}\label{one end ALE}
    If the dimension $n$ is even, then the antipodal map is an orientation-reversing map. Then $\mathbb{S}^n/\{-Id\}$ is non-orientable. Consider the universal cover $\tilde{M}\cong \mathbb{S}^n\setminus\{x,-x\}$, blown up at $x$ and $-x$, then $\overline{M}/\Gamma$ is a Schwarzschild ALE manifold. If $\Gamma$ contains no $\mathbb{Z}_2$-rotation, we can further do a quotient (otherwise, we fix the equator). Then $(M,g)\cong (\mathbb{R}P^n\setminus\{p\},g_{\mathbb{R}P^n})/\Gamma$, which is an LCF, ALE manifold with one end. But it is not orientable. Such a one-end, non-orientable ALE is illustrated in Figure \ref{fig:one end ALE}. We call such a metric a \textbf{non-orientable Schwarzschild ALE}.
\end{example}

\begin{figure}[t]
    \centering
    \resizebox{1\textwidth}{!}{\begin{circuitikz}
\tikzstyle{every node}=[font=\Huge]

% all coordinates shifted by (-110, -10.9)
\draw [color={rgb,255:red,4; green,51; blue,255}, line width=2pt] (9.3,-3.2) -- (9.3,-3.2);
\draw [line width=1pt] (-11.8,11.4) .. controls (-16.5,4.6) and (-16.3,-0.7) .. (-11.8,-6.9);
\draw [line width=1pt] (-11.8,11.4) .. controls (-7.3,4.3) and (-8.3,-0.7) .. (-11.8,-6.9);
\draw [line width=1pt] (-15.3,2.1) .. controls (-13.3,0.6) and (-10.5,0.9) .. (-8.8,2.1);
\draw [line width=1pt, dashed] (-15.3,2.1) .. controls (-13,3.3) and (-10.5,3.3) .. (-8.8,2.1);

\draw [<->, >=Stealth] (-7.8,9.8) .. controls (-2,3.1) and (-2,2.4) .. (-7.8,-5.4)
  node[pos=0.5, fill=white]{$\sigma$};
\draw [->, >=Stealth] (-7.3,11.4) -- (3.8,11.4) node[pos=0.5, fill=white]{Quotient by $\sigma$};

\draw [line width=1pt] (5.3,10.1) .. controls (3.0,7.8) and (2.0,5.4) .. (2.0,1.9);
\draw [line width=1pt] (5.3,10.1) .. controls (7.5,7.3) and (8.8,5.9) .. (8.5,1.9);
\draw [line width=1pt] (2.0,1.9) .. controls (4.5,0.6) and (6.3,0.9) .. (8.5,1.9);
\draw [line width=1pt, dashed] (2.0,1.9) .. controls (4.5,3.1) and (6.3,3.1) .. (8.5,1.9);

\draw [line width=1pt] (-12.8,6.6) .. controls (-12.0,6.1) and (-11.5,6.4) .. (-10.8,6.6);
\draw [line width=1pt] (-13.0,-2.9) .. controls (-12.3,-3.4) and (-11.8,-3.4) .. (-11.0,-2.9);
\draw [line width=1pt] (4.3,5.3) .. controls (5.0,5.1) and (5.5,5.1) .. (6.3,5.3);

\draw [->, >=Stealth] (5.3,0.6) -- (5.3,-4.6) node[pos=0.5, fill=white]{Blow up};

\draw [line width=1pt] (2.3,-11.4) .. controls (4.8,-12.4) and (6.5,-12.4) .. (8.8,-11.4);
\draw [line width=1pt, dashed] (2.3,-11.4) .. controls (4.8,-10.6) and (6.5,-10.6) .. (8.8,-11.4);
\draw [line width=1pt] (2.3,-11.4) .. controls (4.0,-9.1) and (5.5,-8.1) .. (0.8,-5.6);
\draw [line width=1pt] (8.8,-11.4) .. controls (6.8,-8.9) and (5.0,-8.4) .. (9.5,-5.4);
\draw [line width=1pt, dashed] (3.3,-7.1) .. controls (5.3,-6.6) and (5.3,-6.6) .. (7.3,-7.1);
\draw [line width=1pt] (3.3,-7.1) .. controls (5.3,-7.6) and (5.8,-7.6) .. (7.3,-7.1);
\draw [line width=1pt] (4.5,-9.6) .. controls (5.3,-9.8) and (5.3,-9.8) .. (6.0,-9.6);
\draw [line width=1pt] (2.5,-6.6) .. controls (5.3,-7.1) and (5.5,-7.1) .. (7.8,-6.6);
\draw [line width=1pt, dashed] (2.5,-6.6) .. controls (5.0,-6.1) and (6.0,-6.1) .. (7.8,-6.6);

\node [font=\Huge] at (-11.0,-8.4) {$(M,g)$};
\node [font=\Huge] at (11.8,0.1) {$(M/\sigma, g)$};
\node [font=\Huge] at (13.3,-12.9) {$(M/\sigma \setminus \{o\}, G^{\frac{4}{n-2}} g)$};

\end{circuitikz}}
    \caption{The construction of non-orientable Schwarzschild ALE: We start with $(M,g) = (\mathbb{S}^n/\Gamma, g_{\mathbb{S}^n})$. To construct the one-end ALE manifold, we first do the quotient by $\sigma = -\operatorname{Id}$, then, we can blow up the orbifold singularity by the conformal Green's function.}
    \label{fig:one end ALE}
\end{figure}
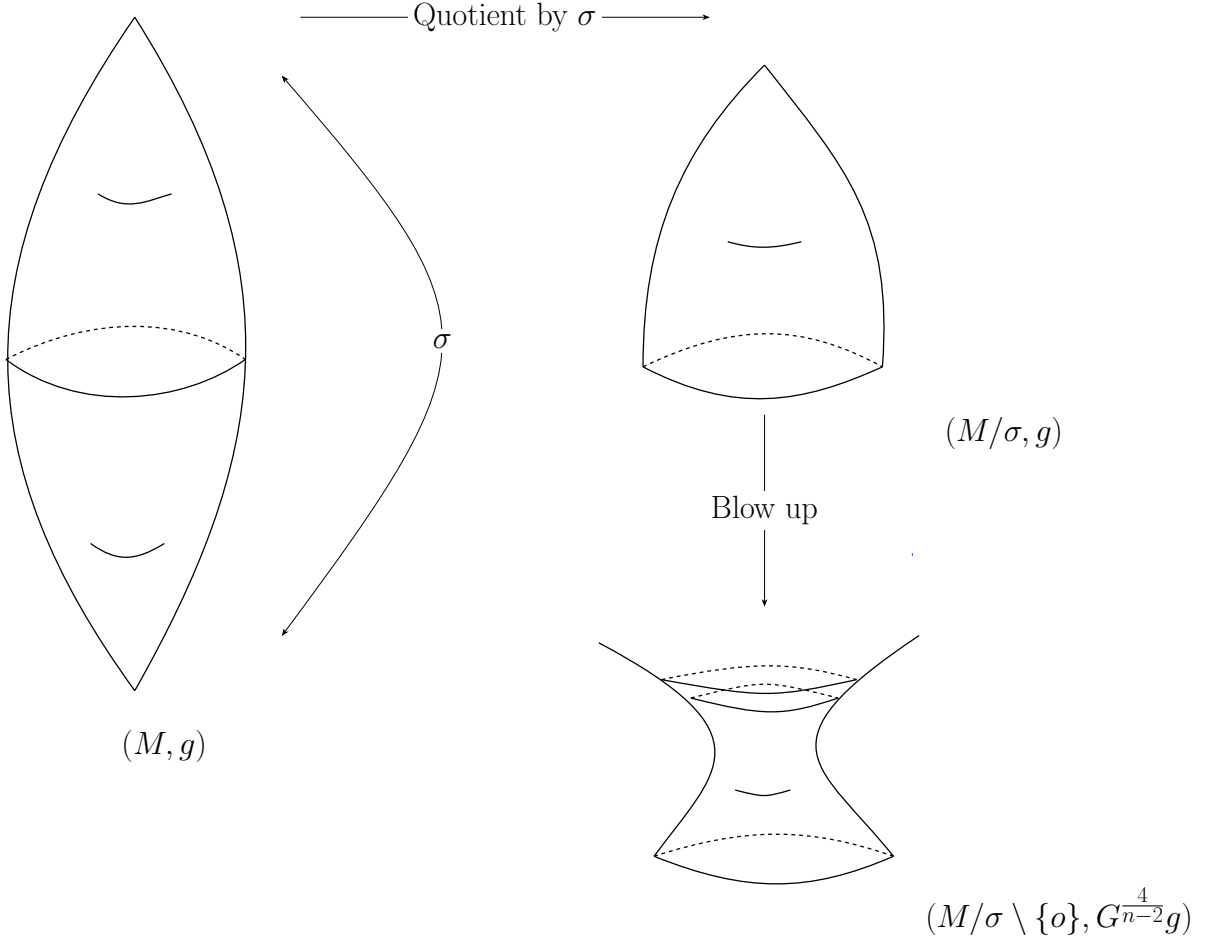

Note that the same construction does not hold when $n$ is odd. The only non-trivial $\Gamma$ is the $\mathbb{Z}_2$ action, then the resulting quotient contains edge singularities.

To end this section, we think that this 'pair-of-ends' phenomenon holds even in the presence of the parabolic element. We conjecture that

\begin{conjecture}
    When $n\geq 5$, for an orientable LCF, nonnegative scalar curvature ALE manifold $(M,g)$, the number of non-trivial ALE ends is even, and they occur in orientation-reversing conjugate pairs.
\end{conjecture}

\section{Application to the moduli space}

\subsection{The oriented moduli space}\label{moduli space application proof}
Now, we consider the moduli space $\mathfrak{M}(n, \mu_0, C_0)$, and $\mathfrak{M}'(n, \mu_0, C_0)$.

From \cite{Aku},\cite{TV2}, $\mathfrak{M}(n, \mu_0, C_0)$ can be compactified under the Gromov-Hausdorff topology by adding the limits of the Gromov-Hausdorff limits. A similar result holds for $\mathfrak{M}'(n, \mu_0, C_0)$. More precisely:
\begin{theorem}[Tian-Viaclovsky]
    Let $\{(M_i,g_i)\}$ be a sequence in $\mathfrak{M}(n,\mu_0, C_0)$, then we have:
    \begin{align*}
        \begin{tikzcd}[ampersand replacement=\&]
         \& (M_i,g_i, p_i) \arrow[r, "GH"] 
         \& (M_\infty, g_\infty, p_\infty)
      \end{tikzcd}
    \end{align*}
    where $(M_\infty, g_\infty, p_\infty)$ is a LCF, $n$-multifold with bounded diameter and a finite set of singularities $S\subseteq M_\infty$. The convergence is smooth away from $S$. 

    Moreover, for every singular point $x\in S$, there exists a sequence $x_i \in M_i$ with curvature $t_i = |\operatorname{Rm}(g_i)|(x_i)\to\infty$ such that the $(M_i, t_i g_i, x_i)$ pointed Gromov-Hausdorff converges to an LCF, scalar-flat, ALE manifold.
\end{theorem}

\begin{remark}
    If we assume $\{(M_i,g_i)\}$ to be orientable, then the orientation of the sequence induces an orientation of the ALE limit, i.e. the limiting ALE manifold is orientable.
\end{remark}

If $n = 3,4$, then, combining Remark \ref{n=3,4} and Corollary \ref{orientable ALE end}, we have the following structure theorem for the multifold limit $(M_o, g_o)$:

\begin{theorem}\label{GH limit}
    Let $n = 3,4$. Let $(M^n_o, g_o)$ be a Riemannian multifold that is a limit of a sequence of closed, orientable, LCF manifolds $(M_i^n, g_i) \in \mathfrak{M}'(n, \mu_0, C_0)$ converging in the above sense, and $S = \{x_1,...,x_k\}$ be its singular set. Assume that the tangent cone for $x\in S$ is:
    \begin{align*}
        T_x M_o = \coprod_{j=1}^m C(\mathbb{S}^{n-1}/\Gamma_j)
    \end{align*}
    Then, the number of non-trivial cones in $T_x M_o$ is even. In particular, when $n =3$, $T_x M_o = \coprod_{j=1}^m \mathbb{R}^n$.
\end{theorem}

Before the proof, we need the $\epsilon$-regularity lemma from \cite{TV2}. See also \cite{And89}, \cite{BKN} for the Einstein case.

\begin{lemma}[$\epsilon$-regularity]\label{epsilon regularity}
    There exists $C(n, S)$, $\epsilon_0=\epsilon_0(n,S)$ such that if 
  \begin{align*}
      \int_{B_r}|\operatorname{Ric}|^{\frac{n}{2}}dv<\epsilon_0,
  \end{align*}
  
  then
  \begin{align*}
      \operatorname{sup}_{B_{\frac{r}{2}}}|\operatorname{Ric}|\leq \frac{C}{r^2}(\int_{B_r}|\operatorname{Ric}|^{\frac{n}{2}}dv)^{\frac{2}{n}}.
  \end{align*}

  Where $S$ is the Sobolev inequality.
\end{lemma}

On the ALE manifold, we have the following gap lemma by letting $r\to \infty$.

\begin{lemma}[Gap lemma for ALE manifolds]\label{gap ALE}
Let $(M,g)$ be an LCF, scalar-flat, ALE manifold with bounded Sobolev constant $S$. There exists $C(n, S)$, $\epsilon_0=\epsilon_0(n,S)$ such that if 
  \begin{align*}
      \int_M|\operatorname{Ric}|^{\frac{n}{2}}dv<\epsilon_0.
  \end{align*}
Then $(M,g)\cong (\mathbb{R}^n , g_{\mathbb{R}^n})$.
\end{lemma}

To deal with the possible ALE orbifolds, we need the following conformal blow-up lemma:

\begin{lemma}\label{ALE orbifold blow-up}
    Let $(M,g)$ be an LCF, nonnegative scalar curvature, ALE orbifold, with finitely many isolated orbifold singularities, $\{q_i\}$. Then, there exists a harmonic conformal factor, $H$, such that $(M\setminus\{q_i\}, H^{\frac{4}{n-2}}g)$ is an ALE manifold of nonnegative scalar curvature.
\end{lemma}
\begin{proof}
    One way to prove it is to use the existence of the Green's function on ALE spaces, see \cite{grigor1999analytic}. Since $(M,g)$ is of nonnegative scalar curvature, the conformal metric $(M\setminus\{q_i\}, H^{\frac{4}{n-2}}g)$ is also of nonnegative scalar curvature. Another way to prove it is to use the harmonic function constructed in Lemma 2.1 of \cite{ju2023conformally}.
\end{proof}

\begin{proof}[Proof of Theorem \ref{GH limit}]
    If $n = 3,4$, non-negative scalar curvature implies $dim_\mathcal{H}(\Lambda) < 1$, thus, Corollary \ref{orientable ALE end} applies.
    
    We will use the bubble tree scaling argument in \cite{Bubble}. See also \cite{TV2}.

    Choose $x\in S$. We will fix $x$ for later discussion. Let $x_i\in M_i$ be such that $\lim\limits_{i\to\infty} x_i = x$, and $B(x_i, r)$ converges to $B(x,r)$ in the Gromov-Hausdorff sense. Moreover, choose $r_\infty > 0$ small such that:
    \begin{align*}
        \sup_{B(x_i, r_\infty)}|\operatorname{Ric}(g_i)| = |\operatorname{Ric}(g_i)|(x_i) \to \infty
    \end{align*}
    and 
    \begin{align*}
        \int_{B(x,r_\infty)} |\operatorname{Ric}(g_\infty)|^{\frac{n}{2}} dv \leq \frac{\epsilon_0}{2}
    \end{align*}
    where $\epsilon_0$ comes from Lemma \ref{epsilon regularity}. 

    Denote $A(r_1, r_2) = B(x, r_2)\setminus B(x,r_1)$ as the annulus region. Choose $r(i)< r_\infty$ such that:
    \begin{align*}
        \int_{A(r(i), r_\infty)}|\operatorname{Ric}(g_i)|^{\frac{n}{2}} dv = \epsilon_0
    \end{align*}
    By the curvature concentration, $r(i)\to 0$. Thus, for $i$ large enough, $r(i)<r_0$, for any $0<r_0 < r_\infty$.

    Note $(A_i(r_0, r_\infty/2), g_i)\to (A_\infty(r_0, r_\infty/2), g_\infty)$ in the Gromov-Hausdorff sense, for some $r_0 >0$. Then, by Lemma \ref{epsilon regularity}, we have $C^{1,\alpha}$ convergence (actually $C^\infty$ combining with the PDE for Ricci curvature in the LCF case), thus, there exists a diffeomorphism $\psi_i: A_i(r_0, r_\infty/2) \simeq A_\infty(r_0, r_\infty/2)$. Unlike the Einstein case, $A_\infty(r_0, r_\infty/2)$ can have finitely many components corresponding to the decomposition of the tangent cone.

    As in \cite{Bubble} and \cite{TV2}, a nested ALE bubble structure emerges. The blow-up sequence $(M_i, r(i)^{-2}g_i, x_i)$ converges to an ALE orbifold $(N,g_N)$ with possibly finitely many multifold singular points. If $(N,g_N)$ is an ALE manifold, then by Corollary \ref{orientable ALE end}, the non-trivial ALE ends appear in pairs. As shown in \cite{Bubble}, each component in $A_i(r_0, r_\infty/2)$ will be arbitrarily close to a portion of the flat cone $ \mathcal{C}(\mathbb{S}^{n-1}/\Gamma)$ corresponding to the ends. Thus, the non-trivial neck components in $A_i(r_0, r_\infty/2)$ also appear in pairs. By the diffeomorphism, the non-trivial cone in $T_xM_o$ also appears in pairs. It is done. The remaining problem is that when $(N,g_N)$ is an ALE orbifold with multifold singularities, Corollary \ref{orientable ALE end} will fail. A counterexample is the football metric conformal blow-up at one singularity. Moreover, the nested resolved gluing ALE may not have non-negative scalar curvature, and Theorem \ref{thm SY} will fail. Thus, we do not adopt this method. Our argument will be an induction from the deepest bubble.

    Starting from the deepest bubble: Due to Lemma \ref{gap ALE}, the bubble tree will only be of finite depth. Let the depth of the bubble tree be $k$. The rescaling sequence $(M_i, r_k(i)^{-2} g_i, x_{ik})$ pointed $C^{1,\alpha}$ converges to an ALE manifold $(N_k, g_k, \hat{x}_k)$, which is the rescaling sequence corresponding to the singularity $x_{k-1}$ at level $k-1$. Like the above argument, since the deepest bubble $(N_k, g_k)$ is an ALE manifold, by Corollary \ref{orientable ALE end}, the non-trivial cone in $T_{x_{k-1}} N_{k-1}$ appears in pairs. This is true for all the singularities in $(N_{k-1}, g_{k-1}, \hat{x}_{k-1})$, which is the pointed Gromov-Hausdorff limit of $(M_i, r_{k-1}(i)^{-2} g_i, x_{i(k-1)})$. 
    
    We claim that the non-trivial ALE ends in the multifold $(N_{k-1}, g_{k-1})$ also appear in pairs. To see this, consider the conformal blow-up metric at all the singularities $S_{k-1}$. We need to use Lemma \ref{ALE orbifold blow-up}. Note this conformal factor is a finite superposition of conformal factors on all the singularities, and for the multifold point, we can detach the multifold point and consider the conformal factor on each singularity with the same asymptotic. Thus, the conformal factor, $G_{k-1}$, is well-defined, and we have:
    \begin{align*}
        (X_{k-1}, g_{k-1})=(N_{k-1}\setminus S_{k-1}, G_{k-1}^{\frac{4}{n-2}}g_{k-1})
    \end{align*}
    is an ALE manifold. In particular, the resulting ALE manifold is scalar-flat. As a result, Corollary \ref{orientable ALE end} can apply. The non-trivial ALE ends of $(X_{k-1}, g_{k-1})$ appear in pairs. Since all the non-trivial ALE ends which are constructed using conformal blow-up associated with the multifold singularities are already paired due to the last step ALE ends argument, the original ends in $(N_{k-1}, g_{k-1})$ also pair. Thus, an easy induction step will show the same result as the bubble is an ALE manifold. 
\end{proof}

\begin{proof}[Proof of Theorem \ref{moduli space application}]
    Theorem \ref{moduli space application} follows directly from Theorem \ref{GH limit}.
\end{proof}

A corollary is that the football metric $\mathbb{S}^4/\Gamma$ cannot be realized as the Gromov-Hausdorff limit of such a sequence. Anderson \cite{anderson2008survey} asked the question: when an Einstein orbifold can be desingularized by Einstein metrics. It is interesting to understand which orbifolds cannot be approximated by a sequence of Einstein manifolds. Biquard \cite{biquard} has identified certain obstructions, and Ozuch \cite{ozuch2022noncollapsed} has shown that $\mathbb{S}^4/\Gamma$ cannot be approximated by a sequence of Einstein manifolds with Eguchi-Hanson metrics bubbles off. Our result is an orientable LCF analog of the Einstein case.  

\begin{corollary}
\label{c:footnon}
    The football metric $\mathbb{S}^4/\Gamma$ cannot be realized as the Gromov-Hausdorff limit for any sequence of closed, orientable, LCF manifolds $(M^4_i,g_i)\in\mathfrak{M}'(4,\mu_0,C_0)$.
\end{corollary}

\subsection{The non-orientable examples}\label{nonorientable construction}
The football orbifolds defined in Definition~\ref{Def football} have exactly 2 orbifold points. There is a generalization that has only 1 orbifold point but is non-orientable. In the following, we assume that $n$ is even.

\begin{definition}[Non-orientable cap]  Let $n$ be even, and $\Gamma \subset \mathrm{SO}(n)$ be a finite subgroup that acts freely on $S^{n-1}$, acting on $\mathbb{R}^{n+1}$ in the first $n$ coordinates, such that the antipodal map on $S^{n-1}$ is not contained in $\Gamma$. Let $\sigma: S^n \rightarrow S^n$ be the antipodal map of $S^n$, and let $\tilde{\Gamma} = \{ \Gamma, \sigma\}$ be the group generated by $\Gamma$ and $\sigma$. Then the non-orientable cap is $S^n/\tilde{\Gamma}$.
\end{definition}

\begin{remark}
\label{r:gfcap} By taking the Green's function metric of the non-orientable cap at the orbifold point, we obtain a non-orientable ALE manifold with one end and group $\Gamma$ at infinity, which is the non-orientable Schwarzschild ALE in Example \ref{one end ALE}. Obviously, since it is non-orientable, this does not contradict Theorem \ref{classification}.
\end{remark}

In contrast to the orientable case, using the non-orientable caps, here we give two constructions showing non-orientable irreducible multifolds can be realized as a non-collapsing limit of non-orientable LCF 4-manifolds. All of the gluing and perturbation techniques are almost identical to those in \cite{joyce2003constant}, so we will omit the details. A slight difference between $\mathfrak{M}(n,\mu_0,C_0)$ is that the construction here is only a constant scalar metric, not necessarily a Yamabe minimizer. But all previous discussions of $\mathfrak{M}(n,\mu_0, C_0)$ also hold for the moduli space of constant scalar metrics with a bounded Sobolev constant and bounded diameter.

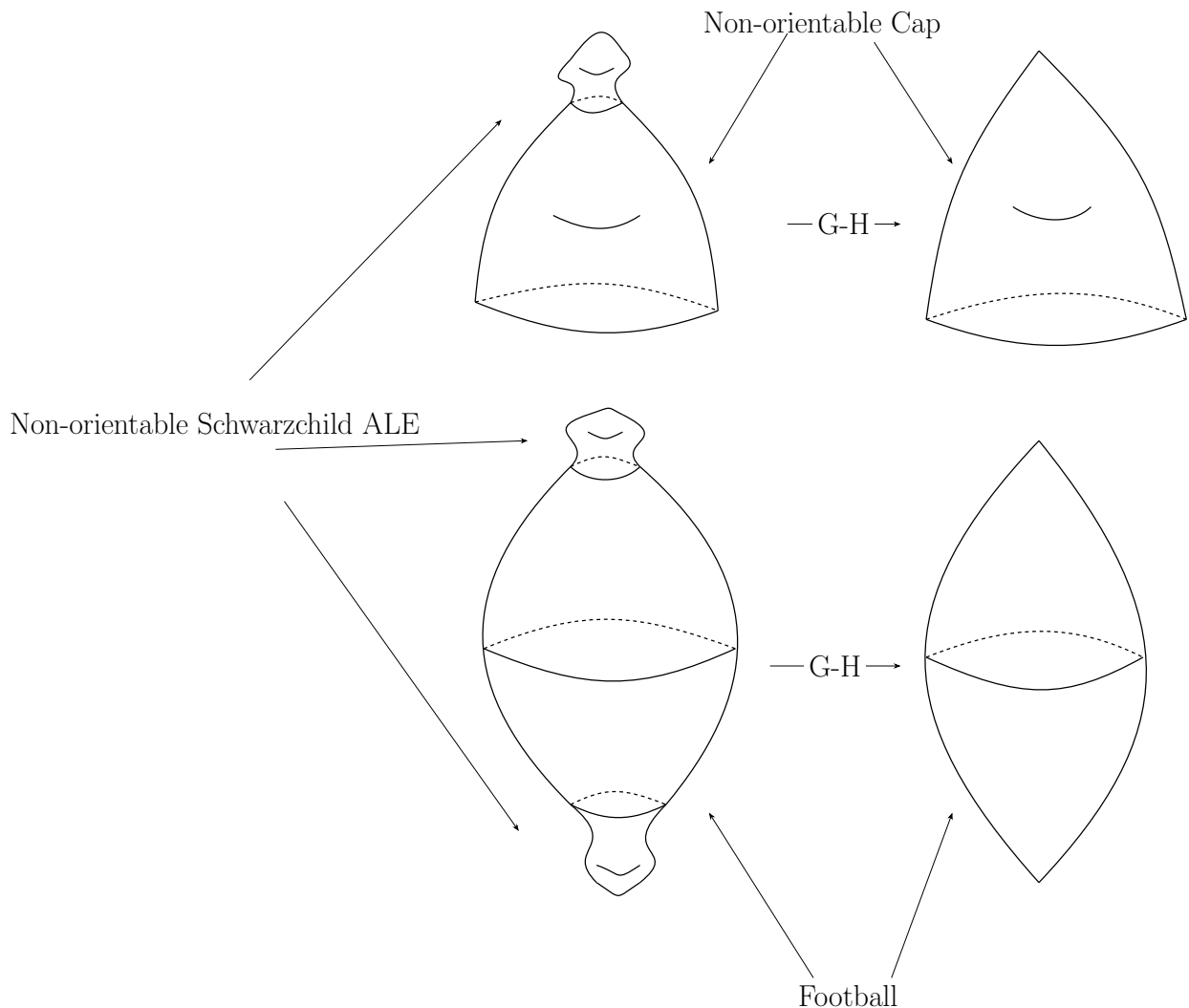
\begin{figure}[t]
    \centering
    \resizebox{1\textwidth}{!}{\begin{circuitikz}
\tikzstyle{every node}=[font=\Huge]

\draw [line width=1pt, short] (-3.75,12) .. controls (-5.5,10.25) and (-6.25,9.25) .. (-6.5,6.25);
\draw [line width=1pt, short] (-2.25,12) .. controls (-0.25,10) and (0.25,9) .. (0.5,6);
\draw [line width=1pt, short] (-6.5,6.25) .. controls (-4,5.25) and (-2.25,5) .. (0.5,6);
\draw [line width=1pt, dashed] (-6.5,6.25) .. controls (-3.5,7) and (-2.25,7) .. (0.5,6);
\draw [line width=1pt, short] (-3.75,12) .. controls (-3.25,11.5) and (-2.75,11.75) .. (-2.25,12);
\draw [line width=1pt, dashed] (-3.75,12) .. controls (-3,12.25) and (-2.75,12.25) .. (-2.25,12);
\draw [line width=1pt, short] (-3.75,12) .. controls (-3.25,13) and (-4.75,12.25) .. (-3.75,13.25);
\draw [line width=1pt, short] (-2.25,12) .. controls (-3,13) and (-1.5,12.75) .. (-2.25,13.5);
\draw [line width=1pt, short] (-3.75,13.25) .. controls (-3,14.25) and (-2.75,14.25) .. (-2.25,13.5);
\draw [line width=1pt, short] (-3.5,13) .. controls (-3,12.75) and (-3,12.75) .. (-2.5,13);
\draw [line width=1pt, short] (6.5,5.75) .. controls (9.25,4.75) and (11.25,4.75) .. (14,5.75);
\draw [line width=1pt, short] (6.5,5.75) .. controls (7,9.25) and (7.5,10.5) .. (9.75,13.5);
\draw [line width=1pt, short] (9.75,13.5) .. controls (12.75,10.5) and (13.25,9.25) .. (14,5.75);
\draw [line width=1pt, dashed] (6.5,5.75) .. controls (9.5,6.75) and (11.5,6.75) .. (14,5.75);
\draw [line width=1pt, short] (-4.25,8.75) .. controls (-3.25,8.25) and (-2.5,8.25) .. (-1.75,8.75);
\draw [line width=1pt, short] (9,9) .. controls (9.75,8.5) and (10.75,8.5) .. (11.25,9);
\draw [line width=1pt, short] (-3.75,1.5) .. controls (-7.25,-2) and (-7,-5) .. (-3.75,-8.25);
\draw [line width=1pt, short] (-1.75,1.5) .. controls (2.25,-2) and (1.5,-5.25) .. (-1,-8.25);
\draw [line width=1pt, short] (-3.75,1.5) .. controls (-3.25,1) and (-2.25,1) .. (-1.75,1.5);
\draw [line width=1pt, short] (-3.75,-8.25) .. controls (-2.75,-8.75) and (-2,-8.75) .. (-1,-8.25);
\draw [line width=1pt, dashed] (-3.75,1.5) .. controls (-2.75,2) and (-2.5,1.75) .. (-1.75,1.5);
\draw [line width=1pt, dashed] (-3.75,-8.25) .. controls (-2.75,-7.75) and (-2.25,-7.75) .. (-1,-8.25);
\draw [line width=1pt, short] (-3.75,1.5) .. controls (-3,2.25) and (-4.75,2.5) .. (-3.25,3);
\draw [line width=1pt, short] (-1.75,1.5) .. controls (-2.5,2.25) and (-0.75,2.25) .. (-2.25,3);
\draw [line width=1pt, short] (-3.25,3) .. controls (-2.5,3.25) and (-2.75,3.25) .. (-2.25,3);
\draw [line width=1pt, short] (-3.25,2.5) .. controls (-2.75,2.25) and (-2.75,2.25) .. (-2.25,2.5);
\draw [line width=1pt, short] (-3.75,-8.25) .. controls (-2.25,-9.5) and (-4,-9.5) .. (-3,-10.5);
\draw [line width=1pt, short] (-1,-8.25) .. controls (-2.5,-9.75) and (-0.5,-9.5) .. (-1.75,-10.5);
\draw [line width=1pt, short] (-3,-10.5) .. controls (-2.25,-11) and (-2.5,-11) .. (-1.75,-10.5);
\draw [line width=1pt, short] (-3,-10) .. controls (-2.25,-10.25) and (-2.5,-10.5) .. (-1.75,-10);
\draw [line width=1pt, short] (-6.25,-3.75) .. controls (-3.25,-5) and (-1.75,-5) .. (1, -3.75);
\draw [line width=1pt, dashed] (-6.25,-3.75) .. controls (-3.25,-2.5) and (-1.75,-2.75) .. (1,-3.75);
\draw [line width=1pt, short] (9.75,2.25) .. controls (5.25,-2.5) and (5.5,-5.75) .. (9.75,-10.5);
\draw [line width=1pt, short] (9.75,2.25) .. controls (14,-3) and (13.75,-6.25) .. (9.75,-10.5);
\draw [line width=1pt, short] (6.5,-4) .. controls (9.25,-5.25) and (10.5,-5.25) .. (12.75,-4);
\draw [line width=1pt, dashed] (6.5,-4) .. controls (9,-3) and (10.75,-3) .. (12.75,-4);

\draw [->, >=Stealth] (2,-4.25) -- (5.75,-4.25) node[pos=0.5, fill=white]{G-H};
\draw [->, >=Stealth] (2.5,8.5) -- (5.75,8.5) node[pos=0.5, fill=white]{G-H};
\draw [->, >=Stealth] (2.5,14) -- (0.25,10.25);
\draw [->, >=Stealth] (5,13.75) -- (7.25,10.25);
\draw [->, >=Stealth] (-13,4) -- (-5.75,11.5);
\draw [->, >=Stealth] (-12.25,2) -- (-5,2.25);
\draw [->, >=Stealth] (-12,0.5) -- (-5.25,-9);

\node [font=\Huge] at (-14,2.75) {Non-orientable Schwarzchild ALE};
\node [font=\Huge] at (3.5,14.25) {Non-orientable Cap};

\draw [->, >=Stealth] (3.25,-13.25) -- (0.25,-8.5);
\draw [->, >=Stealth] (5.5,-13.25) -- (7.25,-8.5);

\node [font=\Huge] at (4.25,-13.75) {Football};

\end{circuitikz}}
    \caption{Two examples of where irreducible multifolds can be realized as a limit in the non-orientable case. The first row illustrates Example \ref{e:non-orientable 1}, where the limit is a non-orientable cap, and one non-orientable Schwarzschild ALE bubbles off. The second row illustrates Example \ref{e: non-orientable 2}, where the limit is a football metric, and two non-orientable Schwarzschild ALEs bubble off. }
    \label{fig:non-orientable moduli}
\end{figure}

\begin{example}[Non-orientable sequence limiting to non-orientable cap]\label{e:non-orientable 1}
Let $(M_1,g_1)$ be the non-orientable Schwarzschild ALE manifold with group $\Gamma$ at infinity as in Remark \ref{r:gfcap}.
Let $(M_2,g_2)$ be the non-orientable cap with only one irreducible $\Gamma$-singularity. We can use gluing techniques to glue the truncated ALE manifold $(M_1,g_1)$ with $(M_2, g_2)$, along the singularity of the neighborhood $\Gamma$ with the gluing parameter $t$. Using the techniques in \cite{joyce2003constant}, for $t$ sufficiently small, we can conformally perturb the glued metric to a constant scalar, LCF, metric with positive scalar curvature. When $t\to 0$, the ALE part $(M_1, g_1)$ bubbles off, and the limit will be isometric to $(M_2,g_2)$. The first row of Figure \ref{fig:non-orientable moduli} illustrates this example.
    
\end{example}

In contrast to Corollary~\ref{c:footnon}, the football metric does actually arise as a Gromov-Hausdorff limit of non-orientable LCF, constant scalar manifolds:

\begin{example}[Non-orientable sequence limiting to orientable football]\label{e: non-orientable 2}
Let $(M_1,g_1)$ be the non-orientable Schwarzschild ALE manifold as in Example \ref{one end ALE}, with $\Gamma$ containing no $\mathbb{Z}_2$-rotation.

Let $(M_2, g_2)$ be the irreducible football orbifold $\mathbb{S}^4/\Gamma$, with the spherical quotient metric. Note $(M_2,g_2)$ is an orientable football metric with two irreducible $\Gamma$-singularities.

We can use gluing and perturbation techniques in \cite{joyce2003constant} to glue two pieces of the ALE manifold $(M_1,g_1)$ with the football metric $(M_2, g_2)$, along the singularity of the neighborhood $\Gamma$ with the same gluing parameter $t$. When $t\to 0$, the two parts of the ALE $(M_1, g_1)$ bubble off, and the limit will be isometric to $(M_2,g_2)$. The second row of Figure \ref{fig:non-orientable moduli} illustrates this example.
\end{example}

\begin{remark}
    Note that the non-orientable one-end ALE metric $(M_1,g_1)$ is an analogue of the $C_\Gamma^\sigma$ caps construction in \cite{ChenZhu}. Although the context is different, we all use the same gluing idea to resolve the orbifold singularity. (They need to cap off the orbifold singularity to run Ricci flow, while we cap off the orbifold singularity to produce a converging sequence.)
\end{remark}

\section{The positive mass theorem for LCF, ALE orbifolds}\label{PMT proof}

Finally, we discuss the positive mass theorem. First, we recall the ADM mass:

\begin{definition}\label{mass}
    Let $(M,g)$ be an ALE manifold. Let the coordinate on one end be $\{x^i\}$ for $r$ large. Then \textbf{the ADM mass} of this end is the following functional:
    \begin{align*}
            m(g):=\lim\limits_{r\to\infty}\frac{1}{2(n-1)\omega}\int_{S_r}(\partial_i g_{ij}-\partial_j g_{ii})\partial_j\llcorner dv,
    \end{align*}
    where $\omega$ is the volume of the unit ball in $\mathbb{R}^n$.
\end{definition}

The positive mass theorem states that in asymptotically flat manifolds with non-negative scalar curvature, with the condition that the scalar curvature is integrable, the mass is non-negative, with zero mass if and only if the asymptotically flat manifold is isometric to $(\mathbb{R}^n, g_{\mathbb{R}^n})$. This was proved by Schoen-Yau \cite{SYPMT1}, \cite{SYPMT2} using minimal surfaces, and by Witten \cite{Witten} using spinors. This was generalized to AE ends of orbifolds in \cite{ju2023conformally}, and to AE ends of spaces with isolated conical singularities in \cite{dai2025positive}.

It is natural to ask the same question on ALE manifolds. However, counterexamples were found by LeBrun \cite{Lebrun}. Thus, a naive extension of the positive mass theorem is not true. Note that LeBrun's counterexamples are anti-self-dual in real dimension 4.

Now, we prove the positive mass theorem in the LCF case:

\begin{proof}[Proof of Theorem \ref{PMT}]
        Our method is an indirect method from \cite{SYconformal}.

        By our assumption, the scalar curvature is integrable hence the mass is well-defined. From Lemma \ref{ALE orbifold blow-up}, we can conformally blow-up the orbifold singularities and form an ALE manifold. From Theorem \ref{positive yamabe}, we can choose the conformal Green's function so that the conformal metric is scalar-flat, and the coordinate such that the ADM mass is well-defined. 

        Let $(M,u^{\frac{4}{n-2}}g)$ to be the above conformal scalar-flat, ALE metric. We want to relate the mass between $(M,u^{\frac{4}{n-2}}g)$ and $(M,g)$. Note that we can adjust the scaling such that $u$ solves:
        \begin{align*}
            \left\{
            \begin{aligned}
                &-\Delta u + c(n) S u = 0\\
                &u\to 1\text{ as }|x|\to\infty
            \end{aligned}
            \right.
        \end{align*}
        Note that $S\geq 0$. The asymptotic of $u$ is:
        \begin{align*}
            u = 1 + Ar^{2-n} + O(r^{2-n-\epsilon})
        \end{align*}
        as $r = |x|\to\infty$. Using Green's formula, Theorem \ref{harmonic on ALE}, and the maximal principle, we have $A=-C\int_MuSdv$, for some $C\geq0$. In particular, $A\leq0$. Then, the mass $m(u^{\frac{4}{n-2}}g) = m(g) + (n-1)A\leq m(g)$. Thus, it is enough to prove the mass is positive for scalar-flat ALE metrics.
        
        By the previous discussion, consider the conformal compactification, $(M',g')$, of $(M,g)$. Consider the manifold cover, $N$, (Proposition \ref{manifold cover}), with the pull-back metric on $N$ is $\overline{g} = P^*g$.
    
    Consider $x_i\in M'$, which is an orbifold point with the local group $\Gamma_i$. The minimal Green's function at $x_i$ is $G_i$. Now, we pull back $G_i$ to $N$. Denote $\overline{G}_i=P^* G_i$. As in \cite{SYconformal}, choose $s\in P^{-1}(x_i)$, there exists a minimal Green's function on $N$, $G_s$, with a singular point at $s$. Thus, $\overline{G}_i\geq G_s$. By the maximal principle of the conformal Laplacian with positive Yamabe constant, we have either $\overline{G}_i>G_s$ or $\overline{G}_i=G_s$.
    
    If $\overline{G}_i>G_s$, then, in the inverted coordinate at $s$, near infinity, 
    \begin{align*}
            \overline{G}_i(x)=1+E_i|x|^{2-n}+O(|x|^{1-n})
    \end{align*}
    and
    \begin{align*}
            G_s(x)=1+E_s|x|^{2-n}+O(|x|^{1-n}),
    \end{align*}
    then we know $E_i>E_s$. Since $P$ is a local isometry, then $m(G_i^{\frac{4}{n-2}}g)>m(G_s^{\frac{4}{n-2}}\overline{g})$. By $N$ being an LCF manifold and by Proposition 4.3 in \cite{SYconformal}, we have $m(G_s^{\frac{4}{n-2}}\overline{g})\geq 0$, thus $m(G_i^{\frac{4}{n-2}}g)>0$.
    
    If $\overline{G}_i=G_s$, then $P^{-1}(x_i)=s$. Thus, all the deck transformations on $N$ fix $s$. Thus, $\Gamma_i\cong\pi_1(M)/\pi_1(N)$. By the previous case, if $m(G_s^{\frac{4}{n-2}}\overline{g})>0$, then it is done. Otherwise, $m(G_s^{\frac{4}{n-2}}\overline{g})=0$, then by Proposition 4.3 in \cite{SYconformal}, $\pi_1(N)=\{e\}$. Thus, $\pi_1(M)\cong \Gamma_i$, which means $N\cong \mathbb{S}^n$ since there is no limiting set. Thus, by Lemma \ref{local group}, $(M',g')$ is conformal to $(\mathbb{S}^n,g_{\mathbb{S}^n})/\Gamma_i$, the football orbifold with the spherical quotient metric. If $\Gamma_i$ is not trivial, $(M,g)$ is either conformal to $(\mathbb{R}^n\setminus\{0\},g_{\text{Sch}})/\Gamma_i$, the Schwarzschild ALE manifolds, in the manifold case; or conformal to $(\mathbb{R}^n, g_{flat})/\Gamma_i$, the flat cone metric, in the orbifold case. By direct computation: In the manifold case: $m(g)>0$. Otherwise, $M'\cong \mathbb{S}^n$. If $m(g)=0$, it is true if and only if $(M,g)\cong (\mathbb{R}^n,g_{\mathbb{R}^n})$. In the orbifold case: $m(g) = 0$ and $(M,g)\cong (\mathbb{R}^n, g_{\mathbb{R}^n})/\Gamma$.
\end{proof}

A direct corollary is a special case of the orbifold Yamabe problem. 

\begin{proof}[Proof of Corollary \ref{orbifold Yamabe corollary}]
        If $(M,g)$ is not conformally equivalent to the football orbifold with the spherical metric, then by Theorem \ref{PMT}, the mass for the conformal Green's function metric at each point $m(\hat{g}_x)>0$.
        
        Now, we choose a singularity $x\in M$. Since $(M,g)$ is LCF, the local distortion constant is purely determined by the mass (See Lemma 9.7 in \cite{lee1987yamabe}), we can use Schoen's test function in \cite{Sch84} and conclude that $Y_{orb}(M,[g])<Y(\mathbb{S}^n)|\Gamma_x|$, where $\Gamma_x$ is the orbifold group. Since this strict inequality holds for any orbifold singularity, we can apply Theorem \ref{t:orbifold Yamabe existence}.  
        
        For the case where $(M,g)$ is conformally equivalent to the football orbifold with the spherical metric, then it is obvious that it contains a metric which is of constant scalar curvature.
\end{proof}

\section{Appendix}
\subsection{The orbifold points in the conformal quotient}\label{orbifold points in the quotient}
We shall prove the following theorem:

\begin{theorem}\label{number ends}
    Let $G\leq C(n)$, a discrete subgroup acting on $\mathbb{S}^n$ properly discontinuously. Denote the limiting set of $G$ as $\Lambda$. Assume that there are no parabolic elements in $G$ and the Hausdorff dimension $\operatorname{dim}_{\mathcal{H}}(\Lambda) < \frac{n-2}{2}$. Suppose there exists a subgroup $\Gamma\leq G$ with isolated fixed points, then the fixed points of the local isotropy subgroup $\Gamma$ must appear in pairs unless the quotient is non-orientable.
\end{theorem}

We need the following lemma from \cite{discrete_group}:
\begin{lemma}(\cite{discrete_group}, pg. 52)\label{hyperbolic fixed point}
  Let $g_1,g_2\in G\leq C(n)$, $g_1$ is hyperbolic. If $g_1$, $g_2$ have precisely one common fixed point, then $G$ is not discrete.
\end{lemma}
\begin{proof}
  Without loss of generality, let the common fixed point be $\infty$, and consider $\mathbb{S}^n=\mathbb{R}^n\cup\{\infty\}$. Then let $g_1(0)=0$, $g_2(0)\neq 0$, and $g_1(\infty)=g_2(\infty)=\infty$. Thus, we can write:
  \begin{align*}
      &g_1(x)=\alpha U_1(x),\\
      &g_2(x)=\beta U_2(x)+b,
  \end{align*}
  for $\alpha,\beta>0$, $\alpha\neq 1$, $b\in\mathbb{R}^n-\{0\}$, $U_1,U_2\in {\rm{O}}(n)$. We can assume that $\alpha<1$. 
  
  Consider 
  \begin{align*}
      h_n(x)&=g_2 g_1^n g_2^{-1}g_1^{-n}(x)\\
      &=U_2U_1^{n}U_2^{-1}U_1^{-n}(x)-\alpha^n U_2 U_1^n U_2^{-1}(b)+b.
  \end{align*}
  If $U_1^m=1$, then let $n=km$, we have
  \begin{align*}
      h_n(x)=x-\alpha^n(b)+b.
  \end{align*}
  Then, let $n\to \infty$, $h_n\to x+b$, which means that $G$ is not discrete.
  
  If $U_1$ has infinite order, then $\forall\epsilon>0$, $\exists n_\epsilon$ such that $|U_1^{n_\epsilon}(x)-x|<\epsilon$, $\forall x\in B(0,r)$. Thus, let $\epsilon\to 0$, we have $h_n\to x+b$, same as before.
\end{proof}

\begin{proof}[Proof of Theorem \ref{number ends}]

Denote $\Omega\subseteq\mathbb{S}^n$ to be the subset where $G$ acts on it properly and discontinuously, $\Lambda = \mathbb{S}^n - \Omega$.

Note that, by Lemma \ref{local group}, $\exists\gamma\in C(n)$ such that $\gamma^{-1}\Gamma\gamma\leq {\rm{O}}(n)$. And $\forall g\in {\rm{O}}(n)$ acting on $\mathbb{S}^n$, by looking at its eigenvectors, if there is one isolated fixed point, then it has two fixed points. We call the other fixed point $-s$, the antipodal point of $s$. Thus, $\gamma^{-1}\Gamma\gamma$ fixes $s$ and $-s$. 

To prove the theorem, we need to show $s$ and $-s$ will be in two different $G$-orbits. Suppose it is not true, we have the following two cases:

\textbf{Case 1}: If $-s\notin \Lambda$

Then $s$ and $-s$ must be in the same orbit. Thus, $\exists\gamma_0\in\rho(\pi_1(M))\leq C(n)$ such that $\gamma_0(s)=-s$.

Set $s=0$, $-s=\infty$, then by Corollary 1.8 in \cite{discrete_group}, we can write:
\begin{align*}
    \gamma_0(v)=r^2 U\frac{v}{|v|^2}+q',
\end{align*}
where $r$ is a constant, $U\in {\rm{O}}(n)$, $q'=\gamma_0^{-1}(\infty)$.

Since the quotient space is an orbifold with local group $\Gamma$, $\gamma^{-1}\Gamma\gamma$ maps the radius $t$, $n-1$ sphere, $S_t\subset \mathbb{R}^n$ to itself. Then, $\gamma^{-1}\Gamma\gamma$ and $\gamma_0^{-1}\gamma^{-1}\Gamma\gamma\gamma_0$ have the same orbit around $-s$. In particular, they fix $S_R\subset\mathbb{R}^n$, for $R$ large. Thus, $\gamma_0$ maps $S_t$ to some $S_R$, which implies that $q'=0$.

Now, $\gamma_0^2(v)=U^2 v$. Since $\gamma_0$ has no fixed point in $\Omega$, but $\gamma_0^2(s)=0=s$, then $\gamma_0^2=\operatorname{Id}$ on $\Omega$. By the assumption, the dimension $\operatorname{dim}(\mathbb{S}^n-\Omega)\leq \frac{n-2}{2}$. Thus, by Theorem \ref{thm L}, $\gamma_0^2=\operatorname{Id}$. Thus, $U^2=\operatorname{Id}$. Then, $U$ is diagonalizable, and all the eigenvalues of $U$ are $\pm 1$. Thus, $U=A^{-1} D_k A$, where $A,D_k\in {\rm{O}}(n)$, $D_k$ is the diagonal matrix with $k+1$ times $1$ and $n-k-1$ times $-1$ in the diagonal. $D_{-1}=-\operatorname{Id}$. It is easy to see that $\gamma_0$ fixes an $\mathbb{S}^k\subset\mathbb{S}^n$. More precisely, let $F$ be the fixed $\mathbb{S}^k$, then $F\subseteq S_r\subseteq\mathbb{R}^n$. We need to force the fixed $\mathbb{S}^k$ to be in the limit set of $G$. In particular, $\operatorname{dim}(\Lambda)\geq k$. Thus, we have the restriction $k<\frac{n-2}{2}$.

Note that we can define a conformal map $\phi:x\mapsto \gamma_1(x)$, for $x\in \mathbb{S}^n$, where $\gamma_1(v)=rA^{-1}(v)$, for $v\in\mathbb{R}^n$. Note also $\gamma_1(s)=s$, $\gamma_1(-s)=-s$. Thus, 
\begin{align*}
    \phi^*(\gamma_0)(v)&=\gamma_1^{-1}\circ\gamma_0\circ\gamma_1(v)\\
    &=r^{-1}A(r^2 A^{-1}D_k A\frac{r A^{-1}v}{r^2|v|^2})=D_k \frac{v}{|v|^2}.
\end{align*}
for $v\in\mathbb{R}^n$. Thus, $\phi^*(\gamma_0)\in {\rm{O}}(n+1)$ by stereographic projection, mapping the equator $\mathbb{S}^{n-1}$ to itself. And for the rest of $g\in \gamma^{-1}\Gamma\gamma\leq {\rm{O}}(n)$, $g(v)=B(v)$, for $B\in {\rm{O}}(n)$, thus,
\begin{align*}
    \phi^*(g)(v)=\gamma_1^{-1}g\gamma_1(v)=ABA^{-1}(v).
\end{align*}
Thus, $\phi^*(g)\in {\rm{O}}(n)$. Hence $\phi^*(\gamma^{-1}\Gamma\gamma)\leq {\rm{O}}(n)$, and it fixes $s$ and $-s$.

Now, we split into different cases based on the number $k$:

\begin{itemize}
    \item $D_{-1}$: In this case $\phi^*\gamma_0=-\operatorname{Id}$.
    
    If the dimension $n$ is odd, then $n-1$ is even, then $\Gamma\cong \mathbb{Z}_2$. Then let $\alpha\in \phi^*(\gamma^{-1}\Gamma\gamma)\leq {\rm{O}}(n)$, then $\alpha$ is a $\mathbb{Z}_2$-rotation around $s$ and $-s$ axis. Then $\phi^*(\gamma_0)\circ \alpha$ fixes the horizontal equator, which is a $\mathbb{S}^{n-1}\subset\mathbb{S}^{n}$. Since $\phi^*\gamma_0\circ\alpha$ corresponds to an element in $\rho(\pi_1(M))$, the resulting quotient will contain an edge singularity, which is a contradiction.

    If the dimension $n$ is even, then $n-1$ is odd. Then $\phi^*(\gamma_0)$ is an orientation-reversing deck transformation by looking at the determinant of $-\operatorname{Id}$. Then $\mathbb{S}^n/\rho(\pi_1(M))$ is non-orientable. In particular, if $\overline{M}\cong \mathbb{S}^n$, then if $\Gamma$ contains no $\mathbb{Z}_2$-rotation, then $(M,g)\cong (\mathbb{R}P^n\setminus\{p\},g_{\mathbb{R}P^n})/\Gamma$, which is diffeomorphic to a non-orientable Schwarzschild ALE manifold as in Example \ref{one end ALE}. But it is not orientable.
    
    \item $D_k$: Since the fixed $\mathbb{S}^k$ is in the limit set $\Lambda$, $\phi^*(\gamma^{-1}\Gamma\gamma)$ maps the fixed $\mathbb{S}^k$ to itself. Hence it corresponds to a group of block matrices, with one block of size $k+1$ and another of size $n-k$. And the $n-k$ matrix will have only one eigenvector as it fixes $s$ and $-s$, and the action on the orthogonal $\mathbb{S}^{n-k-2}$ is free.
    
    Again, when $n$ is odd, then $\gamma^{-1}\Gamma\gamma\cong \mathbb{Z}_2$, then $\phi^*\gamma_0\circ\alpha$ fixes a $\mathbb{S}^{n-k-2}$. It is not contained in $\Lambda$ by the dimension $n-k-2>\frac{n-2}{2}> \operatorname{dim}(\Lambda)$. Hence contradiction.
    
    Thus, when $n$ is even:
    
    If $k$ is even, then $n-k$ is even, $n-k-1$ is odd, then the $n-k$ block has only $-1$ as eigenvalue with multiplicity $n-k-1$. Similarly, choose a $\mathbb{Z}_2$ element $\alpha$, $\phi^*\gamma_0\circ\alpha$ fixes a $\mathbb{S}^{n-k-2}$. Hence contradiction.
    
    If $k$ is odd, then $n-k$ is odd, $n-k-1$ is even. The multiplicity of $-1$ in $\phi^*(\gamma_0)$ is $n-k$ and the multiplicity of $1$ is $k+1$. Then $\phi^*(\gamma_0)$ is orientation-reversing, the quotient is non-orientable.
\end{itemize}

\textbf{Case 2}: If $-s\in\Lambda$

We claim $-s$ is not a hyperbolic fixed point. Suppose $-s$ is a hyperbolic fixed point, then $\exists g \in \rho(\pi_1(M))$, hyperbolic, with $\lim\limits_{n\to\infty} g^n(x) = -s$. By assumption, $s\notin \Lambda$, thus, $\lim\limits_{n\to\infty}g^{-n}(x) = s_0\neq s$. Now, let $\gamma\in \Gamma$. Then $\gamma^{-1}g\gamma \in \rho(\pi_1(M))$ is also hyperbolic. The two fixed points are $-s$ and $\gamma^{-1}(s_0)$. Note $\gamma(s) = s$, then $\gamma^{-1}(s_0)\neq s$. Hence, we can apply Lemma \ref{hyperbolic fixed point}, $\rho(\pi_1(M))$ is not a discrete group, contradiction.
\end{proof}

%%%%%%%%%%%%%%%%%%%%%%%%%%%%%%%%%%%%%%%%%%%%%%%%%%%%%%%
\bibliographystyle{amsalpha}
\bibliography{Xiaokang_geometry_bib}
%\printbibliography
\end{document}